\documentclass[a4paper,11pt,authoryear]{elsarticle}

\usepackage[hidelinks]{hyperref}
\usepackage{amsmath,amsfonts,amssymb,natbib}
\usepackage{fullpage,amstext,amsthm,xspace,pdfsync,enumerate,comment,xcolor,booktabs,xurl}
\usepackage[english]{babel}

\newcommand{\R}{\mathbb R} 
\newcommand{\eps}{\varepsilon}
\newtheorem{theorem}{Theorem}
\newtheorem{remark}{Remark}
\newtheorem{example}{Example}
\newtheorem{proposition}{Proposition}

\newtheorem{corollary}{Corollary}

\allowdisplaybreaks

\makeatletter
\newcommand*{\inlineequation}[2][]{%
  \begingroup
    \refstepcounter{equation}%
    \ifx\\#1\\%
    \else
      \label{#1}%
    \fi
    \relpenalty=10000 %
    \binoppenalty=10000 %
    \ensuremath{%
      #2%
    }%
    \hfill~\@eqnnum
  \endgroup
}
\makeatother

\begin{document}

\begin{frontmatter}
	
	
	
\title{\LARGE Least cores in energy community games}
	

\author[unipi-inf]{Giancarlo Bigi} 
\ead{giancarlo.bigi@unipi.it}
\author[unipi-ing]{Davide Fioriti}
\ead{davide.fioriti@unipi.it}
\author[unipi-inf]{Antonio Frangioni}
\ead{antonio.frangioni@unipi.it}
\author[unimib]{Mauro Passacantando\corref{cor1}}
\cortext[cor1]{Corresponding author}
\ead{mauro.passacantando@unimib.it}
\author[unipi-ing]{Davide Poli}
\ead{davide.poli@unipi.it}

\affiliation[unipi-inf]{organization={Dipartimento di Informatica, University of Pisa},
addressline={Largo B. Pontecorvo 3}, 
city={Pisa},
postcode={56127},
country={Italy}
}

\affiliation[unipi-ing]{organization={Department of Energy, Systems, Territory and Construction Engineering, University of Pisa},
addressline={Largo Lucio Lazzarino}, 
postcode={56122}, 
city={Pisa},
country={Italy}
}

\affiliation[unimib]{organization={Department of Business and Law, University of Milano-Bicocca},
addressline={via Bicocca degli Arcimboldi 8}, 
city={Milan},
postcode={20126}, 
country={Italy}}

\begin{abstract}
An energy community is modeled as a cooperative game, where a veto player is needed beyond the prosumers to manage the community, and the worth of a coalition is its benefit compared to the selfish behaviour of the prosumers. 
Properties of the game such as superadditivity, monotonicity, convexity and balancedness are analyzed both in the presence and absence of admission fees.
Then, the least core and its value are studied in detail, underlying the differences between the cases where the game is balanced or not.
In particular, an exact formula and computable bounds for the least core value are provided, and the maximum and minimum reward in the least core for the veto player are analyzed. 
Finally, a few computational approaches for the exact formula are developed and tested.
\end{abstract}




\end{frontmatter}


\section{Introduction}

The paper deals with a family of coalitional games with transferable utility arising from the analysis of energy communities (see, e.g.,~\cite{Caramizaru2019,gjorgievski2021social}). 
An \emph{Energy Community} (EC) is a group of physically close \emph{prosumers} and a so-called \emph{aggregator}. Prosumers are entities such as households, small/medium companies and public institutions who produce and consume energy at the same time, although not all of them necessarily have both roles. 
In principle, each prosumer that actually produces energy can manage it independently to satisfy their own demand, drawing from the electric grid when needed, and selling excess energy to the grid when appropriate. 
Anyway, joining a local energy community can turn to be more efficient and profitable as governments are providing financial incentives to gather into communities (see, for instance, the rules of the Italian authorities~\cite{arera,gse}). 
Clearly, an energy community needs the overall demand and production of energy to be globally managed: an aggregator is a for-profit entity that monitors and manages energy  exchanges on behalf of the community so that each prosumer can draw from someone else's excess production to satisfy their need, or, vice-versa, provide their excess production to others. Therefore, an optimization problem arises naturally for each possible group of prosumers that are willing to join the aggregator in a community in order to maximize the overall benefit of the resulting EC. 
Similarly, each prosumer can compute their individual benefit via a suitable optimization problem that involves neither any energy exchange nor any other prosumer. 
These optimization problems may range from linear and nonlinear to mixed-integer programs.
For instance, combinatorial or integer variables may be needed to model the behaviour of some generation units and/or the need to acquire some integer amount of equipment such as solar panels and batteries~\citep{FiFP21,FiBiFrPaPo25}.
Yet, these problems may still be relatively easy to solve for the required size via the currently available efficient general-purpose solvers, which implies that their optimal values can be efficiently computed in practice.

Overall, the above setting gives rise to a coalitional Transferable Utility (TU) game, that we name \emph{Energy Sharing Game} (ESG): the players are the prosumers and the aggregator, while the characteristic function measures the overall advantage or disadvantage of a coalition with respect to the self-concerned situation where its members do not form a community. 
As a consequence, the worth of a coalition in ESG is provided by the optimal value of an optimization problem where prosumers control variables, while their presence in the coalition activates constraints and terms in the objective function. 
This framework shares some features with games arising from controlled programming problems \citep{DuSh84} and games of flow \citep{KaZe82b,KaZe82}. 
Indeed, the former games arise from ad hoc convex optimization problems where the players control some constraints, while the latter from network flow problems where they control variables. 
Anyway, ESG may bring in some nonconvexity as well due to the possible presence of fixed admission fees to the community. 
This marks a meaningful difference as the  games in~\citep{DuSh84,KaZe82b,KaZe82} are always totally balanced while ESG may be not. Moreover, the aggregator is a special player who is indispensable for allowing the exchange of energy and thus turns out to be a veto player. 

Since energy communities are awarded a financial incentive as a whole, the non-trivial problem of fairly subdividing their overall benefit among the prosumers and the aggregator arises.
Moreover, any valuable allocation should guarantee some kind of stability of the EC. 
In particular, it should not provide prosumers meaningful reasons to leave the community on the grounds of the possibility of obtaining better benefits on their own.
Tools from cooperative game theory have already been exploited in the literature for energy management, in particular the Shapley value is often used for fair reward sharing~\citep{chics2017coalitional,CREMERS2023120328,TAN2021100453}, but it suffers from stability issues (see, for instance,~\cite{DU2018383,9311637}) and its computation is very demanding.
Alternatively, simple rules that guarantee fairness through egalitarian allocations~\citep{DrFu91,Br07} could be employed though they may not guarantee stability as well.
On the contrary, the \emph{core} is generally considered the set of allocations to choose from in order to guarantee stability~\citep{AbEhLa25,FiBiFrPaPo25,HaMoMc19}.

The paper aims at analyzing stability for ESG, looking for those allocations that maximize the margin between the overall reward of any coalition and its worth.
The set of such allocations is known as the \emph{least core}~\citep{MaPeSh79}, and somehow provides the most stable ones to choose from.
The \emph{nucleolus}~\citep{Sc69} is a specific allocation in the least core, whose computation is viable for very specific classes of games (see, e.g.,~\cite{HaKlSoTiVe03,SoRaTi05}) while extremely demanding in general since it can always be performed through the resolution of either optimization problems of exponentially large sizes or an exponential number of them (see, e.g., \cite{BeFlNg21,NgTh16} and the references therein). 
The nucleolus has already been studied also for general veto games~\citep{ArFe97,Ba19} and some particular subclasses such as clan and big-boss games~\citep{MuNaPoTi88,PoMuTi90,PoPoTiMu89}, which anyhow do not include ESG except for very peculiar simple cases.
Anyway, the main focus of the paper is on the least core as a whole and therefore on the above margin (\emph{least core value}), since it may be worthy to choose an allocation with some further properties beyond stability (see, e.g., the computational approach in~\cite{FiBiFrPaPo25}), while the nucleolus enhances the latter only. 
In particular, the fair share of the reward of the aggregator should be dealt with some additional care, since its role is two-fold: on one side, it enables the formation of the community, while, on the other side, its actual participation is minor with respect to the prosumers. 

The paper is organized as follows. 
Section~\ref{sec:prem} summarizes the framework and main definitions of coalitional TU games, including the class of veto games and its main subclasses.
Section~\ref{sec:eg} introduces the Energy Sharing Game and its formulation as a veto TU game is given. 
A particular attention is given to the simple situation with a unique time step, that allows providing necessary and sufficient conditions for the game to be a clan or big-boss game. 
Sections~\ref{sec:nofees} and \ref{sec:fees} study the properties and the least core of the game without and with fixed admission fees, respectively. 
While the game is always balanced in the former case, this is not necessarily true in the latter.  
An exact formula and simple estimates for the least core value are provided whenever the game is balanced, and estimates only otherwise.  
Furthermore, a linear duality argument is exploited to check when the exact formula holds also in the unbalanced case.
Section~\ref{sec:aggshare} studies the maximum and minimum shares that the aggregator can obtain by the allocations in the least core: while the maximum is known exactly, only lower estimates for the minimum are provided. 
Finally, Section~\ref{sec:algo} addresses the computation of the exact formula for the least core value through suitable families of mixed-integer programs, whose resolution is viable when the underlying optimization problems of the prosumers are linear. 
Five different approaches have been developed and tested for problems up to 200 prosumers,   
and the results show that the fastest is quite efficient.

\section{Preliminaries}
\label{sec:prem}

A cooperative game with transferable utility (TU game) is a pair $(N,v)$, where $N$ is the set of players and $v : 2^N \to \R$, with $v(\emptyset) = 0$, is its \emph{characteristic function}. 
Any nonempty $S \subseteq N$ is a \emph{coalition} and $v(S)$ is its value (or worth).
The set of all players is referred to as the \emph{grand coalition}.
We define the set $\mathcal{P}=\{ S \subseteq N:\ S \neq \emptyset, \ S \neq N \}$ of \emph{proper} coalitions. 
The \emph{marginal contribution} of player $i \in N$ to the grand coalition is denoted by $M^v(i) := v(N) - v(N\setminus i)$. 
An allocation is a vector $x = [ x_i ]_{i \in N} \in \R^{|N|}$ representing the payoffs of the users, and we denote by $x(S) = \sum_{i \in S} x_i$ for any $S \subseteq N$ the total payoff of coalition $S$ where $x(\emptyset)$ is set to 0 by definition.

An allocation $x$ is called imputation if it is efficient, i.e., $x(N)=v(N)$, and individually rational, i.e., $x_i \geq v(i)$ holds for any $i \in N$.
A fundamental concept is that of the \emph{strong $\eps$-core} \citep{ShSh66} of the game $(N,v)$, i.e.,
$
 C_\eps(N,v) =
 \big\{ x \in \R^{|N|}:\ 
		x(N) = v(N), \
		x(S) \geq v(S) + \eps \quad \forall\ S \in \mathcal{P}
        \big\} .
$
The \emph{core}~\citep{Gi59} of the game is $C(N,v) = C_0(N,v)$, i.e., the set of allocations for which no coalition has any incentive to leave the grand coalition. 
A game with non-empty core is called \emph{balanced}. 
It is called \emph{totally balanced} if balancedness holds for any subgame, i.e., whenever the set of all players reduces to a subset of $N$. 

Since not all games are balanced, it is interesting to compute (the largest value of $\eps$ giving rise to) the \emph{least core} \citep{MaPeSh79}, namely $C_{\eps^*}(N,v)$ where $\eps^* 
    = 
    \max \{ \, \eps \,:\, C_\eps(N,v) \neq \emptyset \, \}$
provides the \emph{least core value}.
The \emph{nucleolus} \citep{Sc69} is the particular allocation in the least core that lexicographically minimizes the vector of excesses, $e(S,x)=v(S)-x(S)$, arranged in non-increasing order, for any imputation $x$.
However, the computation of the nucleolus is a very challenging task (see, e.g.,~\cite{BeFlNg21}).
The following classes of TU games are relevant for our discussion:
\begin{enumerate}[a)]
 \item A TU game $(N,v)$ is \emph{superadditive} if  $v(S \cup T) \geq v(S) + v(T)$ for all $S \subseteq N$, $T \subseteq N$ such that $S \cap T = \emptyset$.
 %
 \item A TU game $(N,v)$ is \emph{monotonic} if $v(S) \leq v(T)$ holds for any $S$, $T$ such that $S \subseteq T \subseteq N$; equivalently, $v(S \cup \{ i \}) \geq  v(S)$ for all $S \in \mathcal{P}$ and all $i \in N$. 
 
\item A TU game $(N,v)$ is \emph{convex} if its characteristic function $v$ is supermodular, i.e.,
\begin{equation}
v( S \cup T ) + v( S \cap T ) \geq v( S ) + v( T ) \quad \forall S \subseteq N , T \subseteq N .
\label{eq:convex}
\end{equation}
Monotonicity follows from superadditivity if $v(i) = 0$ for all $i \in N$.
Convex games are always superadditive and a full characterization of the core is available (see, e.g.,~\cite{MaZaSo20}).

\item A TU game $(N,v)$ is a \emph{veto game} if there exists a subset $T \in \mathcal{P}$ such that
\begin{equation}
v(S) = 0 \mbox{ for any } S \subseteq N \mbox{ such that } S \cap T \neq T;
\label{eq:veto}
\end{equation}

\item A monotonic veto game $(N,v)$ is a \emph{clan game} \citep{PoPoTiMu89} if the \emph{veto set} $T \subset N$ satisfying \eqref{eq:veto} also satisfies
\begin{equation}
\textstyle
v(N) - v(S) \geq \sum_{i \in N \setminus S} M^v(i) \mbox{ for any } S \supseteq T;
\label{eq:clan}
\end{equation}

 \item A clan game $(N,v)$ is a \emph{big-boss game} \citep{MuNaPoTi88} if the \emph{clan} $T \subset N$ satisfying \eqref{eq:veto} and \eqref{eq:clan} is a singleton $T = \{ i^* \}$. 
\end{enumerate}
Big-boss games are always superadditive, a full characterization of the core is available, and a closed formula for the nucleolus exists~\citep{MuNaPoTi88}.
Clearly, a big-boss game is a clan game but not vice versa (see, e.g., Proposition~\ref{prop:ex1-bbg-clan} a)). 
The nucleolus of a clan game can be computed by a polynomial algorithm that exploits only marginal contributions, while a full characterization of the core is available for this class of games as well~\citep{Ba19,PoPoTiMu89}.
A general veto game has no particular properties such as~\eqref{eq:clan} that can enforce similar results. 
Anyway, exponential algorithms to compute the nucleolus~\citep{ArFe97,Ba19} and alternative inequality constraints that describe the core~\citep{Ba16} have been designed.

The Energy Sharing Game introduced in the next section is a veto game, but not necessarily a big-boss or clan game except in very peculiar situations. 
Therefore, the corresponding results for the computation of the core and nucleolus cannot be exploited. 



\section{The Energy Sharing Game}
\label{sec:eg}

We model the energy community described in the introduction as a cooperative game with transferable utility whose set of players has the structure $N = U \cup \{ a \}$.
The set $U$ represents the \emph{users} (prosumers), while $a$ is the so-called \emph{aggregator} that allows cooperation between the users.
Each user $i \in U$ is characterised by the following elements:
\begin{itemize}

\item two different sets of variables $z_i \in \R^{n_i}$ and $y_i = [y^+_i , y^-_i] \in \R^{m}_+ \times \R^{m}_+$.
The former describes the internal production, consumption and management of energy and it is always freely available to the user.
The latter describes the amounts of energy that the user shares with the community: it can be positive in case a coalition is formed with some other users and necessarily the aggregator, otherwise it is necessarily fixed to 0. 
In particular, $y^+_i$ is the amount sent to the other users and $y^-_i$ is the amount received from the other users for each of the $m$ time steps of the time horizon.
 
 
\item a compact region $D_i \subseteq \R^{n_i} \times \R^{2m}_+$ that describes the feasible choices of the user, with the property that the section $D^z_i = \{ z_i \in \R^{n_i}: \ (z_i, 0) \in D_i\}$ is nonempty, i.e., it is feasible for the user not to belong to any coalition;
 
 
\item a continuous function $f_i:\R^{n_i+2m} \to \R$ that represents the individual benefit of the user $i$ for its own choices only;

\item an admission fee $c_i \geq 0$ that user incurs for joining any coalition.
\end{itemize}
When user $i$ is not part of any coalition and therefore can not rely on any contribution from the others, his maximum individual benefit is given by 
\begin{equation}
\label{eq:Pi}
\tag{$P_i$}
    b_i = \max \left\{  f_i(z_i,0): \ z_i \in D^z_i \right\}.
\end{equation}
When a coalition $S$ with at least two users and the aggregator is formed, its  maximum overall benefit is obtained by solving the following problem $(P_S)$:
\begin{subequations}
\begin{align}
b_S = \max\ & 
\sum_{i \in S \cap U}[ f_i(z_i,y_i) - c_i ] +
f_a \left( \sum_{i \in S \cap U} y^+_i \right) &
\label{eq:venabler-obj} 
\\
 & (z_i, y_i) \in D_i & i \in S \cap U
\label{eq:venabler-ind} 
\\
& 
\sum_{i \in S \cap U} y^+_i = \sum_{i \in S \cap U} y^-_i,
\label{e:balance}
\end{align}
\label{eq:venabler}
\end{subequations}
where $f_a : \R^m \to \R_+$ is a continuous function that 
represents a reward for the coalition due to the amount of shared energy and therefore we assume  $f_a(0) = 0$.
The constraint~\eqref{e:balance} links together the sharing variables of the different users by providing the overall balancing of the exchanged energy.
Notice that the solution $[(z^*_i, 0)]_{i \in S \cap U}$, with $z^*_i$ the optimal solution of $(P_i)$, is feasible for $(P_S)$: thus the maximum benefit of the coalition is not less than the sum of the maximum individual benefits, at least in the important special case with null admission fees (see Theorem~\ref{t:c=0properties} a)). 



The above optimization problems lead naturally to the cooperative Energy Sharing Game (ESG), whose characteristic function is given by
\begin{equation}
\label{eq:vSdef}
 v(S) = 
 \begin{cases}
 \textstyle b_S - \sum_{i \in S \cap U} b_i & \text{if } S \in \mathcal{P}^a_+ \cup N,
  \\
  0                              & \text{otherwise},
 \end{cases}
\end{equation}
where $\mathcal{P}^a_+ = \{S \in \mathcal{P}^a:\ |S \cap U| \geq 2 \}$ with $\mathcal{P}^a = \{S \in \mathcal{P}:\ a \in S\}$. 
In this way, $v(S)$ represents the increase or decrease in utility due to joining the coalition that the users in $S$ experience together with respect to what they would achieve going alone. 
Notice that the aggregator $a$ is a veto player due to the fact that no energy sharing is possible without $a$.

 


The above game is quite complex to analyze in full generality. 
Properties such as monotonicity and situations such as the presence of clans or big-bosses can be well studied when it has some further structure. 
In particular, we consider the following simple ESG that involves a unique time step and therefore can serve also as a building block of more complex situations.

\

\noindent\textbf{Simple ESG (SESG).}
\emph{
Consider ESG with a unique time step so that it is reasonable to assume that each prosumer acts just as a producer or a consumer.
Let $U_1$ be the set of producers, $U_2$ the set of consumers and $U=U_1 \cup U_2$.
Assume $D_i = Z_i \times Y_i$ for each user $i \in U$, with 
$Z_i \subset \R^{n_i}$ be arbitrary and $Y_i \subset\R^2$ given by: $Y_i=[0, p_i] \times \{ 0 \}$ if $i \in U_1$ and $Y_i=\{ 0 \} \times [0, q_i]$ if $i \in U_2$.
In plain words, each producer $i \in U_1$ can send up to a maximum amount $p_i>0$, while each consumer $i \in U_2$ can receive up to a maximum amount $q_i>0$.
Assume the benefit function has the form
\[
f_i(z_i,y_i) = 
\begin{cases}
  \hat{f}_i(z_i) - \beta y^+_i & \text{if $i \in U_1$}, 
  \\
  \hat{f}_i(z_i) + \alpha y^-_i & \text{if $i \in U_2$},
 \end{cases}
\]
where $\hat{f}_i: \R^{n_i} \to \R$ provides the benefit of user $i$ while not taking part in the coalition, $\alpha$ is the unitary saving of receiving energy from the community, and $\beta$ the unitary cost of sending energy to the community. 
The reward for any coalition is $\gamma>0$ for any exchanged unit within the coalition, i.e., $f_a(t) = \gamma t$. 
Assume $\alpha + \gamma - \beta > 0$, i.e., there is a coalition-wide benefit from exchanging (energy) even when taking into account the cost for sending and the benefit for receiving. 
In this framework the objective function~\eqref{eq:venabler-obj} reads
\[
\sum_{i \in S \cap U} [ \hat{f}_i(z_i) - b_i ]
+ \alpha \sum_{i \in S \cap U_2} y^-_i
- \beta \sum_{i \in S \cap U_1} y^+_i
+ \gamma \sum_{i \in S \cap U} y^+_i 
- \sum_{i \in S \cap U} c_i,
\]
which is separable in $z=(z_i)$ and $y=(y_i)$.
Since $b_i$ is the maximum of $\hat{f}_i$ on $Z_i$, the terms $\hat{f}_i(z_i)-b_i$ are non-positive and their maximum is zero.
Considering~\eqref{e:balance}, the  optimal value of~\eqref{eq:venabler} coincides with the optimal value of the function $(\alpha+\gamma-\beta) \sum_{i \in S \cap U_1} y^+_i - \sum_{i \in S \cap U} c_i$ over the same region.
Hence, whenever $S \in \mathcal{P}^a_+$ one has
\begin{align}\label{e:vS-ex1}
v(S) =
(\alpha+\gamma-\beta) \min \left\{ \, \sum_{i \in S \cap U_1} p_i \,,\,    
       \sum_{i \in S \cap U_2} q_i
    \, \right\}
- \sum_{i \in S \cap U} c_i,
\end{align}
whereas $v(S) = 0$ otherwise by the very definition~\eqref{eq:vSdef}. Relying on the convention that the sum over the empty set is 0, if $S \cap U_i= \emptyset$ for $i = 1$ or $i = 2$ (i.e., there are no producers or no consumers, and therefore no transfer is possible), then $v(S) = - \sum_{i \in S \cap U} c_i$.
Without any loss of generality, $\alpha+\gamma-\beta=1$ will be assumed in all the examples of the paper. 
}

In general, the ESG is not convex as the following example of a simple ESG shows.

\begin{example}
\label{ex:non-convex}
Consider SESG with $U_1 = \{ 1, 2, 3 \}$, $U_2 = \{ 4 \}$, $p_i = i$, $q_4 = 11/2$, $c_i = c \geq 0$ for any $i \in U$.
The values of some meaningful coalitions are: $v(\{ a , 1 , 2 , 4 \})=3-3c$, $v(\{ a , 1 , 3 , 4 \})=4-3c$, $v(\{ a , 1 , 4 \})=1-2c$, $v(N) = 11/2 - 4c$.
The game is not convex since for $S = \{ a , 1 , 2 , 4 \}$ and $T = \{ a , 1 , 3 , 4 \}$
one has: 
$v(S \cup T) + v(S \cap T) 
= 13/2 - 6c 
< 7-6c 
= v(S) + v(T)$.
\end{example} 

In simple ESGs monotonicity is strictly related to fixed costs and the structure of producers and consumers.



\begin{proposition}
\label{prop:ex1-mono}
Consider the Simple Energy Sharing Game.
\begin{enumerate}[a)]

\item Assume $|U_i| \geq 2$ for any $i=1,2$. 
Then, SESG is monotonic if and only if $c_j=0$ for any $j \in U$.

\item Assume $U_1=\{u_1\}$. 
Then, SESG is monotonic if and only if
\begin{align}
\label{e:ex1-mono}
& c_{1} \leq \min\left\{ p_1 , \ \min_{j \in U_2} q_j \right\}
\qquad
\text{and}
\qquad
c_j = 0 \quad \forall\ j \in U_2.
\end{align}

\item Assume $U_2=\{u_2\}$. 
Then, SESG is monotonic if and only if
\begin{align*}
& c_{2} \leq \min\left\{ q_1 , \ \min_{j \in U_1} p_j \right\}
\qquad
\text{and}
\qquad
c_j = 0 \quad \forall\ j \in U_1.
\end{align*}

\end{enumerate}
\end{proposition}

\begin{proof}
a) If the costs are all zero, then monotonicity follows immediately from \eqref{e:vS-ex1}.
Vice versa, if SESG is monotonic, then $0=v(S) \leq v(S \cup \{k\})=-c_j-c_k$ for any coalition $S=\{a, j\}$, with $j,k \in U_i$ and $j \neq k$. 
Hence, $c_j=c_k=0$ follows since they are both non-negative.

\noindent b) Suppose that \eqref{e:ex1-mono} holds and consider any $S \in \mathcal{P}^a$ and $T \supset S$.
If $u_1 \notin T$, then $v(S)=v(T)=0$.
If $u_1 \in T \setminus S$, then
$
v(T) 
=
\min\left\{ p_1 , \ \sum_{j \in T \cap U_2} q_j \right\} - c_{1}
\geq
\min\left\{ p_1 , \ \min_{j \in U_2} q_j \right\} - c_{1}
\geq 0=v(S).
$
If $u_1 \in S$, then
$
v(T) 
=
\min\left\{ p_1 , \ \sum_{j \in T \cap U_2} q_j \right\} - c_{1}
\geq
\min\left\{ p_1 , \ \sum_{j \in S \cap U_2} q_j \right\} - c_{1}
= v(S).
$
If $a \notin S$, then the required inequality follows similarly. Therefore, SESG is monotonic.

Vice versa, if SESG is monotonic, then the same argument as in case a) shows $c_j=0$ for any $j \in U_2$. Moreover, 
$
0
=
v(\{a,u_1\})
\leq
v(\{a,u_1,j\})
=
\min\{ p_1, q_j \} - c_{1} 
$
holds for any $j \in U_2$. Hence, condition \eqref{e:ex1-mono} holds.

\noindent c) Analogous to b).
\end{proof}

When SESG is monotonic, a characterization of the presence of a clan or a big-boss is available relying on the amounts of production and consumption only.

\begin{proposition}
\label{prop:ex1-bbg-clan}
Consider a monotonic Simple Energy Sharing Game.

%
\begin{itemize}
 \item[a)] If $U_1=\{u_1\}$ (resp. $U_2=\{u_2\}$), then the SESG is a clan game, where the clan is given by $\{a, u_1\}$ (resp. $\{a, u_2\}$).
\end{itemize}

Further, assume $|U_i| \geq 2$ for any $i=1, 2$.
Then, $a$ is the only veto player and the following statements hold:
\begin{itemize}
\item[b)] If SESG is a big-boss game, then $\sum_{j \in U_2} q_j \neq \sum_{j \in U_1} p_j$.

\item[c)] If $\sum_{j \in U_2} q_j < \sum_{j \in U_1} p_j$, then SESG is a big-boss game if and only if 
\begin{equation}
\textstyle
\sum_{j \in U_2} q_j \leq \sum_{j \in U_1 \setminus i } p_j,
\qquad \forall\ i \in U_1,
\label{e:bbg}
\end{equation}
or equivalently $M^v(i)=0$ holds for any $i \in U_1$.
 
 \item[d)] If $\sum_{j \in U_2} q_j > \sum_{j \in U_1} p_j$, then SESG is a big-boss game if and only if
       \[
        \textstyle
        \sum_{j \in U_1} p_j \leq \sum_{j \in U_2 \setminus i} q_j,
        \qquad \forall\ i \in U_2,
       \]
 or equivalently $M^v(i)=0$ holds for any $i \in U_2$.

\end{itemize}
\end{proposition}

\begin{proof}
a) There are two veto players: $a$ and $u_1$. 
Let $S = \{ a, u_1 \} \cup W$, where $W \subseteq U_2$. Then, the monotonicity of the game guarantees that \eqref{e:ex1-mono} holds. Hence, we get 
$v(N) = \min\left\{ \, p_1 \,,\, \sum_{j \in U_2} q_j \, \right\}-c_{1}$;
$v( N \setminus \, i) 
= \min\left\{ \, p_1 \,,\, \sum_{j \in U_2 \setminus  i } q_j \, \right\}-c_1$, for all $i \in U_2$; $v(S) = \min\left\{ \, p_1 \,,\,  \sum_{j \in W} q_j \, \right\}-c_{1}$.
We distinguish three cases:
If $p_1 \geq \sum_{j \in U_2} q_j$, then $
       v(N) - v(S) =
       \sum_{i \in U_2} q_i -  \sum_{i \in W} q_i =
       \sum_{i \in U_2 \setminus W} q_i =
       \sum_{i \in N \setminus S} M^v(i)$.

If $p_1 < \sum_{j \in U_2} q_j$ and $p_1 \leq \sum_{j \in W} q_j$, then $v(N) = v(S) = p_1 - c_{1}$.  Moreover, for all $i \in U_2 \setminus W$ we have $\sum_{j \in U_2 \setminus i} q_j \geq \sum_{j \in W} q_j \geq p_1$, hence $M^v(i) = 0$ for any $i \in U_2 \setminus W$ and $v(N) - v(S) = 0 = \sum_{i \in N \setminus S} M^v(i)$.

If $p_1 < \sum_{j \in U_2} q_j$ and $p_1 > \sum_{j \in W} q_j$, then $v(N)=p_1-c_{1}$ and $v(S)=(\sum_{j \in W} q_j) -c_{1}$. Moreover, for all $i \in U_2 \setminus W$ we have
$$
M^v(i) 
= 
p_1 - \min\left\{ p_1 , \sum_{j \in U_2 \setminus i } q_j \right\} 
=    
\max\left\{ 0 , p_1 - \sum_{j \in U_2 \setminus i } q_j \right\}.
$$
Let $K = \{ \, i \in U_2 \setminus W \,:\, p_1 > \sum_{j \in U_2 \setminus i } q_j \, \}$. If $K = \emptyset$, then $M^v(i) = 0$ holds for all $i \in U_2 \setminus W$, and
$v(N) - v(S) = p_1 - \sum_{j \in W} q_j > 0 =
\sum_{i \in N \setminus S} M^v(i)$. 
If $|K| \geq 1$, then the following chain of equalities and inequalities holds:
\begin{align*}
\sum_{i \in N \setminus S} M^v(i) =  & 
\sum_{i \in K} \left( p_1 -\sum_{j \in U_2 \setminus i } q_j \right) 
=
\sum_{i \in K} \left( p_1 +q_i-\sum_{j \in U_2} q_j \right) 
\\
=  & 
|K|\left( p_1 - \sum_{j \in U_2} q_j \right) + \sum_{i \in K} q_i
\leq |K|\left( p_1 - \sum_{j \in U_2} q_j \right) + \sum_{j \in U_2 \setminus W} q_j
\\
= & 
|K|\left( p_1 - \sum_{j \in U_2} q_j \right) + 
\sum_{j \in U_2} q_j - \sum_{j \in W} q_j 
\\
=  & 
(|K|-1)\left( p_1 -\sum_{j \in U_2} q_j \right) +
p_1 - \sum_{j \in W} q_j 
\leq \, 
\textstyle
p_1 - \sum_{j \in W} q_j = v(N) - v(S).
\end{align*}
Therefore, SESG is a clan game where $\{a,u_1\}$ is the clan.

\noindent b) Assume, by contradiction that $\sum_{j \in U_2} q_j=\sum_{j \in U_1} p_j$. 
Then, we have $v(N)=\sum_{j \in U_2} q_j$, $M^v(i) = q_i$ for all $i \in U_2$ and $M^v(i) = p_i$ for all $i \in U_1$. Hence,~\eqref{eq:clan} implies
$
\textstyle
v(N) = v(N) - v( \{ \, a \, \} ) \geq
\sum_{j \in U} M^v(j) =
\sum_{j \in U_1} M^v(j) + \sum_{j \in U_2} M^v(j) =
2v(N),      
$
that is $v(N)=0$, which is impossible.

\noindent c) The assumption guarantees $v(N) = \sum_{j \in U_2} q_j$, $M^v(i) = q_i$ for all $i \in U_2$ and $M^v(i) = 0$ for any $i \in U_1$. 
The coalition $S = \{ a \}$ satisfies~\eqref{eq:clan}: in fact,
$v(N) - v(\{a\}) = v(N) = \sum_{j \in U_2} q_j =
 \sum_{j \in U} M^v(j)$.
For any other coalition $S \supset \{a\}$ we have
\begin{align*}
 v(N) - v(S) =
 &  \textstyle
    \sum_{j \in U_2} q_j 
    - \min\left\{ \sum_{j \in S \cap U_1} p_j \ , \ \sum_{j \in S \cap U_2} q_j \right\} 
    \\
 = & \textstyle
 \max\left\{ \sum_{j \in U_2} q_j - \sum_{j \in S \cap U_1} p_j \ , \
             \sum_{j \in U_2 \setminus (S \cap U_2)} q_j
      \right\} 
      \\
 \geq & \textstyle
 \sum_{j \in U_2 \setminus (S \cap U_2)} q_j = \sum_{j \in U_2 \setminus (U_2 \cap S)} M^v(j) \\
 = & \textstyle
 \sum_{j \in U_1 \setminus (U_1 \cap S)} M^v(j) +
 \sum_{j \in U_2 \setminus (U_2 \cap S)} M^v(j)
 = \sum_{j \in N \setminus S} M^v(j). 
\end{align*}
Therefore, SESG is a big-boss game. 
Vice versa, if SESG is a big-boss game, then
\[
v(N) = v(N) - v( \{ a \} ) \geq
\sum_{j \in U} M^v(j) =
\sum_{j \in U_1} M^v(j) + \sum_{j \in U_2} M^v(j) =
\sum_{j \in U_1} M^v(j) + v(N),      
\]
that implies $M^v(i) = 0$ for any $i \in U_1$, so that
$
    \sum_{j \in U_2} q_j = v(N) = v(N \setminus i) 
    =\linebreak
    \min 
    \left\{
    \sum_{j \in U_2} q_j , \sum_{j \in U_1 \setminus i} p_j
    \right\}$ for any $i \in U_1$. Hence condition~\eqref{e:bbg} holds.

\noindent d) Analogous to c).
\end{proof}

Notice that Proposition
\ref{prop:ex1-bbg-clan} a) actually provides a full characterization of clan games that are not big-boss games, namely $\min\{|U_1|, |U_2|\}=1$.
With the aid of Proposition~\ref{prop:ex1-bbg-clan} we can illustrate that ESG is unrelated to clan/big-boss games:
\begin{itemize}
 %
 \item Assume $|U_1| \geq 2$, $|U_2|\geq 2$, $c_i=0$ for any $i \in U$ and $p_i \geq \sum_{j \in U_2} q_j$ for any $i \in U_1$: the only veto player is $a$ and SESG is a big-boss game since condition~\eqref{e:bbg} holds.

 \item Assume $|U_1| \geq 2$, $|U_2|\geq 2$, $c_i=0$ for any $i \in U$ and $p_i=p$ for any $i \in U_1$, with $|U_1|-1 < (\sum_{j \in U_2} q_j)/p < |U_1|$.
 The only veto player is $a$ and SESG is not a big-boss game since condition~\eqref{e:bbg} does not hold.

 %

 %
\end{itemize}


\section{The least core for games without admission fees}
\label{sec:nofees}

In this section, we consider ESG in the case of no admission fees, i.e., when $c_i = 0$ for all $i \in U$. This makes the analysis somewhat simpler, since the game enjoys some useful properties (see Theorem~\ref{t:c=0properties} below).
In order to stress the difference with the general case, we denote the characteristic function of the game by $v_0$ and its least core value by $\eps^*_0$. 
Furthermore, an allocation is ordered as $x=(x_1,\dots,x_{|U|},x_a)$, i.e., the last component represents the share of the  aggregator.

\begin{theorem}
\label{t:c=0properties}
The ESG has the following properties:
\begin{enumerate}[a)]
\item $v_0(S) \geq 0$ for any $S \subseteq N$

\item the game is superadditive

\item the game is monotonic

\item the game is balanced and $(0,\dots,0,v_0(N)) \in C(N,v_0)$.
\end{enumerate}
\end{theorem}

\begin{proof}
a) Since $[(z^*_i, 0)]_{i \in S \cap U}$, with $z^*_i$ the optimal solution of $(P_i)$, satisfies the constraints of problem $(P_S)$ and $c_i=0$ for any $i \in U$, we have $b_S \geq \sum_{i \in S \cap U} b_i$.
 
 \noindent b) Let $S,T \subset N$ such that $S \cap T = \emptyset$. If $a$ belongs to neither $S$ nor $T$, then $v_0(S) + v_0(T) = 0 = v_0(S \cup T)$. If $a \in S$, $|S \cap U| \geq 2$ and $a \notin T$, then we denote by $[ \, z^*_i \,,\, y^*_i \, ]_{i \in S \cap U}$ the optimal solution of $(P_S)$ and by $[ \, z^*_i \,,\, 0 \, ]$ the optimal solution of $(P_i)$ for all $i \in T$. The union of previous vectors is obviously a feasible solution for $(P_{S \cup T})$ since
 $$
\textstyle
 \sum_{i \in (S \cup T) \cap U} (y^*_i)^+ = \sum_{i \in S \cap U} (y^*_i)^+
 = \sum_{i \in S \cap U} (y^*_i)^-
 = \sum_{i \in (S \cup T) \cap U} (y^*_i)^-.
$$
Hence, 
\begin{align*}
\textstyle
 v_0(S \cup T) \geq & 
\textstyle
 \sum_{i \in (S \cup T) \cap U}[ f_i(z^*_i,y^*_i) - b_i ] 
	+ f_a \left( \sum_{i \in (S \cup T) \cap U} (y^*_i)^+ \right) 
 \\[3mm]
 = &
\textstyle
 \sum_{i \in S \cap U}[ f_i(z^*_i,y^*_i) - b_i ] + f_a \Big(\sum_{i \in S \cap U} (y^*_i)^+ \Big) +
 \sum_{i \in T}[ f_i(z^*_i,0) - b_i ] 
 \\[3mm]
 = & \, v_0(S) + 0 = v_0(S) + v_0(T). 
\end{align*}
\noindent c) From a) and b), $S \subseteq T \subseteq N$ yields $v_0(T) = v_0(S \cup (T \setminus S)) \geq v_0(S) + v_0(T \setminus S) \geq v_0(S)$.

\noindent d) Since $x(N) = v_0(N)$ and for any $S \subset N$ we have $x(S) = v_0(N) \geq v_0(S)$ if $a \in S$ and $x(S) =0 = v_0(S)$ otherwise. Thus, 
we can conclude that $x = (0,\dots,0,v_0(N)) \in C(N,v_0)$ and hence the game is balanced.
\end{proof}

The peculiar structure of ESG is not essential for the game to be balanced, indeed any monotonic veto game is balanced (see~\cite{Ba16}).
Moreover, any subgame is still a monotonic veto game and therefore balanced, so the ESG is totally balanced (see Corollary~\ref{cor:totbal}).
 
The allocation in Theorem~\ref{t:c=0properties}d) is extremely unfair and hardly of use in practice: $v_0$ measures the extra value that the coalition bring to the users as opposed to being alone, 
thus this allocation gives no user an incentive to either leave or join the grand coalition. 
In order for the grand coalition to form, users should obtain a positive benefit in the first place. 
Thus, an allocation in some  strong $\eps$-core with $\eps > 0$, possibly in the least core, should be selected.
The next theorem provides a characterization of the nonemptiness of strong $\eps$-cores and consequently the least core value.
Since the results in the theorem are largely common with the case of nonzero admission fees, their proof is delayed until the corresponding results in the next section.

\begin{theorem}
\label{t:c=0eps-core}
\ \begin{enumerate}[a)]
\item For any $\eps \geq 0$, the following statements are equivalent:
\begin{enumerate}[(i)]
\item $C_\eps(N,v_0) \neq \emptyset$

\item $(\eps,\dots,\eps,v_0(N)-\eps|U|) \in C_\eps(N,v_0)$

\item $\eps \leq \min \left\{ \,  (v_0(N)-v_0(S))/(|N|-|S \cap U|) \,:\, S \in \mathcal{P}^a \, \right\}$.
\end{enumerate}
\item The exact value of $\eps^*_0$ is
\begin{align}
\eps^*_0 = \min \big\{ \, ( \, v_0(N)-v_0(S) \, ) \,/\, 
( \, |N|-|S \cap U| \, ) \,:\, 
S \in \mathcal{P}^a \, \big\}.
\label{e:leastcore-exact}
\end{align}

\item The following lower and upper bounds for $\eps^*_0$ hold:
\begin{align}
\label{e:leastcore-lb}
& \eps^*_0 \geq \underline\eps_0 =
( \, v_0(N) - \max\{ \, v_0(S) \,:\, S \in \bar{\mathcal{P}}^a \, \} \, )
\,/\, |N|,
\\[3mm]
\label{e:leastcore-ub}
& \eps^*_0 \leq \bar\eps_0 =
\min \big\{ v_0(N)\,/\,|N| \,,\,
( \, v_0(N) - \max\{ \, v_0(S) \,:\, S \in \bar{\mathcal{P}}^a \, \} \, )
\,/\, 2 \big\},
\end{align}
where $\bar{\mathcal{P}}^a = \{ \, S \in \mathcal{P}^a \,:\, |S| = |N|-1 \, \}$. 
\end{enumerate}
\end{theorem}

Notice that the monotonicity of the game and \eqref{e:leastcore-exact} imply that $\eps^*_0=0$ if and only if the marginal contribution $M^{v_{{}_0}}(i)=0$ for some $i \in U$. 
When $\eps=0$, $(iii)$ collapses to $v_0(S) \leq v_0(N)$ for any $S \in \mathcal{P}^a$ as already shown in~\cite{ArFe97} and \cite{PoMuTi90}.

\begin{remark}
\label{rem:bbg-clan-c0}
Consider SESG when it is a big-boss game.
Then, Proposition~\ref{prop:ex1-bbg-clan} c) or d) implies that the marginal contribution $M^{v_{{}_0}}(i)=0$ for any $i \in U_1$ or $i \in U_2$, so that $\eps^*_0 = 0$.

When SESG is a clan game with $U_1=\{u_1\}$ (the case with $U_2=\{u_2\}$ is symmetric), there are two cases: 
if $p_1 \geq \sum_{j \in U_2} q_j$, then $\eps^*_0 >0$ since all the marginal contributions are positive. Moreover, \eqref{e:leastcore-exact} provides
$
\eps^*_0 =
\min\left\{ \frac{\sum_{j \in U_2} q_j}{|U_2|+2} , \ 
\min\left\{
\frac{\sum_{j \in S} q_j}{|S|+1}: \ 
S \subset U_2, \
S \neq \emptyset
\right\}
\right\}.
$
If, in addition, $q_j=q$ for any $j \in U_2$, then
$\eps^*_0 = q/2$.
If $p_1 < \sum_{j \in U_2} q_j$, then either $\eps^*_0=0$ or $\eps^*_0>0$ can occur. 
Indeed, \eqref{e:leastcore-exact} reads
$
\eps^*_0 = 
\min\left\{
\frac{p_1}{|U_2|+2} , \ 
\min\left\{
\frac{\left( p_1 -\sum_{j \in S} q_j \right)_+}{|U_2|+1-|S|}
: \ S \subset U_2, \ S \neq \emptyset
\right\}
\right\},
$
where $(\cdot)_+$ denotes the positive part.
In addition, when $q_j=q$ for any $j \in U_2$, then 
$
\eps^*_0=
\left(
\frac{p_1 - q(|U_2|-1)}{2}
\right)_+.
$
\end{remark}

The Example below shows that the bounds~\eqref{e:leastcore-lb} and \eqref{e:leastcore-ub} are tight.

\begin{example}
\label{ex:c0-lbub-tight}
Consider SESG with $U_1=\{1,2\}$, $U_2=\{3,4\}$, $p_1=10$, $p_2=90$, $q_3=86$, $q_4=14$, $c_i=0$ for any $i \in U$.  
The non-zero values of coalitions are: 
$v_0(\{ a , 2 , 3 , 4 \})=90$;
$v_0(S) \leq 86$ if $a \in S$, $|S| = 4$ and $S \neq \{ a , 2 , 3 , 4 \}$;
$v_0(\{ a , 2 , 3 \})=86$;
$v_0(S) \leq 14$ if $a \in S$, $|S|= 3$ and $S \neq \{ a , 2 , 3 \}$,
while $v_0(N)=100$. 
We have $\eps^*_0 = 14/3$, while \eqref{e:leastcore-lb} gives $\underline\eps_0=2$ and ~\eqref{e:leastcore-ub} gives $\bar\eps_0 = 5$. 
If instead it were $p_2 =92$, then one would have $\eps^*_0 = \bar\eps_0 = 4$. The lower bound in~\eqref{e:leastcore-lb} is tight if and only if $\eps^*_0 = 0$, or equivalently there exists a user $i \in U$ such that $v_0(N) = v_0( N \setminus i )$. For instance, if it were $p_2=100$, then $v_0(N) = v_0(\{ a , 2 , 3 , 4\}) = 100$.
\end{example}

Notice that the upper and lower bounds~\eqref{e:leastcore-lb}--\eqref{e:leastcore-ub} require the computation of $|N|$ values of the characteristic function, i.e., the solution of an optimization problem of the form $(P_S)$ for each value. 
If these problems are linear programs, then the bounds can be computed in polynomial time (see Section~\ref{sec:algo}). 


\section{The least core for games with admission fees}
\label{sec:fees}

The case with (positive) admission fees $c_i$ is more complex. 
In fact, the properties enjoyed by the game with no fees do not hold in general (see Example~\ref{ex:c>0no_properties}) and the core might be empty if fees are not small enough (see Theorem~\ref{t:c>0core}).
Notice that the characteristic function can be also written as
\[
 v(S) =
 \begin{cases}
    v_0(S) - c(S) & \text{if } S \in \mathcal{P}^a_+,
    \\
    0    & \text{otherwise},
 \end{cases}
\]
where $v_0$ is the characteristic function of the game with no fees and $c(S) = \sum_{i \in S \cap U} c_i$.
The core of this game shares some similarities with the  weak $\eps$-core~\citep{ShSh66} of $(N,v_0)$: if all the costs $c_i$ are equal to some $c$ and $\eps=-c$, then the inequality constraints for the coalitions $S \in \mathcal{P}^a_+$ coincide (but not for the other coalitions), anyway the two sets have empty intersection because of the efficiency of allocations since $v(N)<v_0(N)$.



\begin{example}
\label{ex:c>0no_properties}
Consider SESG with $U_1=\{1,2\}$, $U_2=\{3,4\}$, $p_1=4$, $p_2=6$, $q_3=2$, $q_4=7$, $c_i = 1$ for all $i \in U$.
Then:
\begin{enumerate}[a)]
\item $v(S) < 0$ for some $S$: for instance, $v(\{ a , 1 , 2 \})=-2$,

\item the game is not superadditive: for instance, $v(\{ 3 \}) + v(\{ a , 1 , 4 \}) = 2 > 1 = v(\{ a , 1 , 3 , 4 \})$,

\item the game is not monotonic: for instance, $v(\{ a , 1 , 4 \}) = 2 > 1 = v(\{ a , 1 , 3 , 4 \})$,

\item the game is not convex: for instance, $S = \{ a , 1 , 3 , 4 \}$ and $T = \{ a , 2 , 3 , 4 \}$ give $       v(S \cup T) + v(S \cap T) 
       = 3 < 4 = v(S) + v(T)$.
\end{enumerate}
\end{example}

Monotonicity is related to the magnitude of admission fees. 
The following theorem provides a characterization through two conditions: the admission fee of each user must not exceed its marginal contribution to any coalition (when fees are not taken into account), and the total admission fee of the users in any coalition must not exceed  the extra value provided by the aggregator to the coalition.

\begin{theorem}
\label{t:mono}
The ESG is monotonic if and only if both the following conditions hold:
\begin{enumerate}[(i)]
 \item $c_i \leq  \min\{v_0(S\cup \{ i \}) - v_0(S) \,:\, S \in \mathcal{P}^a \;,\; i \notin S \}$
       for all $i \in U$

 \item $c(S) \leq v_0(S \cup \{ a \}) -v_0(S)$ for all $S \in \mathcal{P}$ such that $a \notin S$. 
\end{enumerate}
\end{theorem}

\begin{proof}
\noindent The monotonicity is equivalent to ask $v(S \cup \{ i \}) \geq v(S)$ for all $S \in \mathcal{P}$ and all $i \in N$. The two conditions correspond to the cases where, respectively, $i$ is a user or $i = a$. In the former, consider $S \in \mathcal{P}$ such that $i \notin S$. If $a \notin S$, then $v(S \cup \{ i \}) = 0 = v(S)$. If $a \in S$, then
\begin{align*}
 v(S \cup \{ i \}) \geq  v(S) \Longleftrightarrow & \; \textstyle
 v_0(S \cup \{ i \}) - c(S) - c_i \geq v_0(S) - c(S) \\
 \Longleftrightarrow & \; v_0(S \cup \{ i \}) - c_i \geq v_0(S)
 \Longleftrightarrow c_i \leq v_0(S \cup \{ i \}) - v_0(S)
\end{align*}
that is equivalent to $(i)$. The other case is that of $S \in \mathcal{P}$ and $a \notin S$ (thus $v(S) = 0$), where $v(S \cup \{ a \}) \geq v(S)$ is the same as $v_0(S \cup \{ a \}) \geq c(S)$.
\end{proof}

While monotonicity guarantees the nonemptiness of the core, the vice versa does not hold. In fact, balancedness can be characterized by a limited form of monotonicity (condition $(iii)$ below, see also~\cite{ArFe97}) or bounds on admission fees.

\begin{theorem}
\label{t:c>0core}
The following statements are equivalent:
\begin{enumerate}[(i)]
 \item $C(N,v) \neq \emptyset$
 
 \item $(0,\dots,0,v(N)) \in C(N,v)$

 \item $v(S) \leq v(N)$ for all $S \in \mathcal{P}^a$
 
 \item $c(U) - c(S) = c(U \setminus S) \leq
        v_0(N) - v_0(S)$ for all $S \in \mathcal{P}^a$.
\end{enumerate}
\end{theorem}

\begin{proof}
\hfill \\
$(ii) \Longrightarrow (i)$ is obvious.

\noindent
$(i) \Longrightarrow (iii)$. Assume that there exists $x \in C(N,v)$: for all $i \in N$ we have $x_i \geq v( \{ i \} ) = 0$. 
Thus, for any $S \in \mathcal{P}^a$, $v(N) = x(N) 
\geq x(S) \geq v(S)$.

\noindent
$(iii) \Longrightarrow (iv)$. Take any $S \in \mathcal{P}^a$. If $|S \cap U|\geq 2$, then $v_0(N) - c(U) = v(N) \geq v(S) = v_0(S) - c(S)$ as required. If, instead, $|S \cap U| \leq 1$, then $v_0(N) - c(U) = v(N) \geq v(S) = 0 = v_0(S)$. Thus, $v_0(N) \geq v_0(S) + c(U) \geq v_0(S) + c(U \setminus S)$ since $U \setminus S \subseteq U$ and $c_i \geq 0$.

\noindent
$(iv) \Longrightarrow (ii)$. With $\bar{x}=(0,\dots,0,v(N))$, $\bar{x}(N) = v(N)$ and $\bar{x}(S) = 0 = v(S)$ for all $S \subset N$ such that $a \notin S$. If, rather, $S \subset N$ and $a \in S$, then $\bar{x}(S) = v(N) = v_0(N) - c(U) \geq 
 v_0(S) + c(U \setminus S) - c(U) =
 v_0(S) - c(S) = v(S)$.
Therefore, $\bar{x} \in C(N,v)$.
\end{proof}

\begin{remark}
Denoting the maximum admission fee by $\bar{c}$, the inequality
\[
 \bar{c} \leq (v_0(N)-v_0(S)) / (|N|-|S \cap U| - 2)
 \qquad \forall\ S \in \mathcal{P}^a
\]
implies condition $(iv)$ in~\autoref{t:c>0core} and therefore the nonemptiness of the core since $c(U \setminus S) \leq (|U|-|S|) \bar{c} = (|N|-|S \cap U| - 2)\bar{c}$.
The above bound on $\bar{c}$ is slightly larger than $\eps^*_0$ from~\eqref{e:leastcore-exact},
which shows that the maximum admission fee is tightly related to the least core value for the case with no admission fees.
\end{remark}

Notice that condition $(i)$ from~\autoref{t:mono} implies condition $(iv)$ in \autoref{t:c>0core} and hence it is sufficient for balancedness.
Monotonicity is stronger than $(i)$, indeed it guarantees the nonemptiness of the core of any subgame.

\begin{corollary}
\label{cor:totbal}
$(N,v)$ is totally balanced if and only if it is monotonic.
\end{corollary}

\begin{proof}
If $(N,v)$ is monotonic, then condition $(iii)$ in \autoref{t:c>0core} holds for any subgame and thus the game is totally balanced.
Vice versa, consider that any subgame $(T,v)$ is balanced. If $a \in T$, then \autoref{t:c>0core} $(iii)$ implies $v(S) \leq v(T)$ for any $S \in \mathcal{P}^a$ with $S \subset T$. 
In particular, it holds $v(T) \geq v(\{a\})=0$ and therefore $v(T) \geq v(S)=0$ for any $S \notin \mathcal{P}^a$ with $S \subset T$.
Finally, if $a \notin T$ then $v(S)=v(T)=0$ for any $S \subseteq T$. Therefore, the game is monotonic. 
\end{proof}

Notice that the ESG in Example~\ref{ex:c>0no_properties} is not totally balanced but it is balanced (it is easy to check that the restricted monotonicity condition $(iii)$ in~\autoref{t:c>0core} holds).

Just like in the case with no fees (see~\autoref{t:c=0eps-core}), the extremely unfair allocation where the aggregator takes all characterizes the balancedness of the core.
On the contrary, the opposite extreme solution where the users take all does not necessarily belong to the core. 
The egalitarian allocation~\citep{Br07} and the center of gravity~\citep{DrFu91} are built to be fairer and they coincide in our framework by providing the aggregator and the users with an even share of the overall benefit.
Proposition~\ref{prop:extsol} provides necessary and sufficient conditions for them to belong to the core.

\begin{proposition}
\label{prop:extsol}
The following statements hold:
\begin{enumerate}[a)]

\item $x^{ua}=(v(N)/|U|,\dots,v(N)/|U|,0) \in C(N,v)$ if and only if 
\begin{align}\label{e:egalU}
\text{$v(N) \geq 0$ \quad and \quad $\max\{ v(S)/|S \cap U|:\ S \in \mathcal{P}^a_+ \} \leq v(N)/|U|$.}
\end{align}

\item $x^{eg}=(v(N)/|N|,\dots,v(N)/|N|) \in C(N,v)$ if and only if 
\begin{align}\label{e:egalN}
\text{$v(N) \geq 0$ 
\quad
and 
\quad
$\max\{ v(S)/|S|:\ S \in \mathcal{P}^a_+ \} \leq v(N)/|N|$.}
\end{align}

\item If $x^{ua} \in C(N,v)$, then $x^{eg} \in C(N,v)$.

\end{enumerate}
\end{proposition}

\begin{proof}
\hfill \\
a) Clearly, $x^{ua}$ satisfies $x^{ua}(N) = |U| ( \, v(N) \,/\, |U| \, ) = v(N)$. 
Since $x^{ua}(S) = |S \cap U| \, v(N) / |U|$ and $v(S)=0$ for any $S \notin \mathcal{P}^a_+$, the equivalence follows from the definition of the core.

\noindent b) It is analogous to a).

\noindent c) If $x^{ua} \in C(N,v)$, then~\eqref{e:egalU} holds for a).
For any $S \in \mathcal{P}^a_+$ we have
\begin{align*}
\frac{v(S)}{|S|} 
& =
\frac{v(S)}{|S \cap U|} \, \frac{|S \cap U|}{|S|} 
=  
\frac{v(S)}{|S \cap U|} \, \frac{|S|-1}{|S|}
\leq
\frac{v(N)}{|U|} \, \frac{|S|-1}{|S|}
\\
& = \frac{v(N)}{|N|-1} \, \frac{|S|-1}{|S|}
=
\frac{v(N)}{|N|} \, \frac{|N|}{|N|-1} \, \frac{|S|-1}{|S|} 
\leq 
\frac{v(N)}{|N|} \, \frac{|N|}{|N|-1} \, \frac{|N|-1}{|N|} 
=
\frac{v(N)}{|N|},
\end{align*}
where the first inequality follows from~\eqref{e:egalU} and the second one holds since the function $f(t)=(t-1)/t$ is increasing for $t>0$. Hence, condition~\eqref{e:egalN} holds, i.e., $x^{eg} \in C(N,v)$. 
\end{proof}

Notice that \eqref{e:egalU} and \eqref{e:egalN} cannot hold if some user has a zero marginal contribution, which is always the case in simple ESGs that are big-boss games.
As expected, both conditions hold for simple ESGs where the number of producers and consumers is the same and their individual production and consumption are all equal.
Conditions~\eqref{e:egalU} and \eqref{e:egalN} are not equivalent as the following example shows.

\begin{example}
Consider SESG with $U_1=\{1,2\}$, $U_2=\{3,4\}$, $p_1=4$, $p_2=7$, $q_3=4$, $q_4=6$, $c_i=0$ for any $i \in U$.
The coalition $\{a,2,4\}$ shows that \eqref{e:egalU} does not hold since $v(S)/|S \cap U|=3>5/2=v(N)/|U|$. On the contrary, \eqref{e:egalN} holds since $v(N)/|N|=2$ and $v(S) \leq 6$ and $|S| \geq 3$ for any $S \in \mathcal{P}^a_+$.
\end{example}


\subsection{The least core in the balanced case}

Supposing that the core is nonempty, characterizations of the nonemptiness of strong $\eps$-cores can be obtained also with the presence of admission fees.
A formula for the least core value and computable upper and lower bounds follow as well.


\begin{theorem}
\label{t:c>0eps-core}
\begin{enumerate}[a)]
\item For any $\eps \geq 0$ the following statements are equivalent:
\begin{enumerate}[(i)]
 \item $C_\eps(N,v) \neq \emptyset$
 
 \item $(\eps,\dots,\eps,v(N) - \eps|U|) \in C_\eps(N,v)$

\item $\eps \leq \hat{\eps}$
\end{enumerate}

\noindent where $\hat\eps :=
\min \left\{ [v(N)-v(S)]/(|N|-|S \cap U|):\ S \in \mathcal{P}^a
\right\}$.

\item If $C(N,v) \neq \emptyset$, then the least core value is $\eps^*=\hat\eps$.

\item If $C(N,v) \neq \emptyset$, then $\eps^*_0 - \bar{c} < \eps^* \leq \eps^*_0$ (with $\bar{c} = \max_{i \in U} c_i$).
\item If $(N,v)$ is monotonic, then
      \begin{equation}
       \label{e:c>0-lbmono}
       \eps^* \geq \big( \, v(N) - \max \{ v(S):\ S \in \bar{\mathcal{P}}^a \, \} \, \big) \,/\, |N|
       \; .
      \end{equation}
\item If $C(N,v) \neq \emptyset$, then the following upper bound for $\eps^*$ holds:
\begin{equation}
\label{e:c>0-ub}
\eps^* \leq 
\min \left\{ \eps^*_0 \;,\;
\frac{v(N)}{|N|} \,,\
\frac{v(N) - \max\{v(S):\ S \in \bar{\mathcal{P}}^a \}}{2} 
\right\}.
\end{equation}
\end{enumerate}
\end{theorem}
\begin{proof}
\begin{enumerate}[a)]
\item $(ii) \Longrightarrow (i)$ is obvious.

\noindent
 $(i) \Longrightarrow (iii)$. Assume  there exists $x \in C_\eps(N,v)$: by definition, $x(N) = v(N)$ and $x(S) \geq v(S) + \eps$ holds for any $S \subset N$. Thus, we have $x_i \geq \eps$ for all $i \in N$. If $S \in \mathcal{P}^a$, then 
\begin{align*}
 v(N) = & \, 
 x(N) = x(S) + x(N \setminus S) \geq
 x(S) + \eps |N \setminus S| \geq 
 v(S) + \eps (1 + |N \setminus S|) \\
 = & \, 
 v(S) + \eps (1 + |N| - |S|) =
 v(S) + \eps (|N| - |S \cap U|)
\end{align*}
hence, $\eps \leq \hat\eps$ holds.

\noindent
$(iii) \Longrightarrow (ii)$: for $\bar{x} = (\eps,\dots,\eps,v(N)-\eps|U|)$, $\bar{x}(N)=v(N)$ by definition. For $S \subset N$ such that $a \notin S$ (hence $v(S) = 0$) one has $\bar{x}(S) = \eps |S| = v(S) + \eps |S| \geq v(S) + \eps$. If, rather, $S \subset N$ and $a \in S$, then
\begin{align*}
 \bar{x}(S) = & \,
 \eps |S \cap U| + v(N) - \eps |U| = v(N) - \eps (|U|-|S \cap U|) \\
 \geq & \,
 v(S) + \eps (|N|-|S \cap U|) - \eps (|U|-|S \cap U|) \\
 = & \,
 v(S) + \eps (1 + |U|-|S \cap U|) - \eps (|U|-|S \cap U|) = v(S) + \eps,
\end{align*}
where the inequality follows from assumption $(iii)$. Hence, $\bar{x} \in C_\eps(N,v)$.
\item It follows from $a)$.
\item From b) and \eqref{e:leastcore-exact} we have
\begin{align*}
\eps^*_0 - \bar{c} & = \min \left\{ \frac{v_0(N)-v_0(S)}{|N|-|S \cap U|}
              \,:\, S \in \mathcal{P}^a
	   \right\}
 - \bar{c}
 \\
 & <  \min \left\{ \frac{v_0(N)-v_0(S) - |N \setminus S| \bar{c}}{|N|-|S \cap U|}
              \,:\, S \in \mathcal{P}^a
	   \right\}
\\
& \leq
\min \left\{ \frac{v_0(N) - v_0(S) - c(N \setminus S)}{|N|-|S \cap U|}
\,:\, S \in \mathcal{P}^a
\right\} = \eps^*
\\
& \leq
\min \left\{ \frac{v_0(N)-v_0(S)}{|N|-|S \cap U|} \,:\, S \in \mathcal{P}^a
\right\}
=
\eps^*_0.
\end{align*}

\item The least core value in b) and monotonicity imply
\begin{align*}
\eps^*  
    & \geq
    \min \left\{ \frac{v(N)-v(S)}{|N|}:\ S \in \mathcal{P}^a \right\}
	\\
	& 
	= 
	\frac{v(N) - \max \left\{ v(S) \,:\, S \in \mathcal{P}^a \right\}}{|N|}
= \frac{v(N) - \max \left\{ v(S):\ S \in \bar{\mathcal{P}}^a \right\}}{|N|}.
\end{align*}

\item $\eps^* \leq \eps^*_0$ follows from c). 
For the other two terms in the minimum, the least core value in b) guarantees:
$\eps^* 
\leq
\frac{v(N)-v(\{ e \})}{|N|-|\{ e \} \cap U|} = \frac{v(N)}{|N|}$
and
\begin{align*}
\eps^* 
\leq
\min \left\{ \frac{v(N)-v(S)}{2}
\,:\, S \in \bar{\mathcal{P}}^a
\right\} 
=
\frac{v(N) - \max \{ v(S):\ S \in \bar{\mathcal{P}}^a \}}{2}.
\end{align*}
\end{enumerate}
\end{proof}

The only relevant difference with the case of no admission fees (see Theorem~\ref{t:c=0eps-core}) lies in the upper bounds~\eqref{e:leastcore-ub} and \eqref{e:c>0-ub} since the latter involves the least core value $\eps^*_0$. 
Anyway, a looser bound can be achieved by removing $\eps^*_0$ from \eqref{e:c>0-ub} that, on the other hand, makes its computation  affordable.
The Example below shows that the bounds~\eqref{e:c>0-lbmono} and \eqref{e:c>0-ub} are tight.

\begin{example}
\label{ex:c-lbub-tight}
With the data of Example~\ref{ex:c0-lbub-tight} and $c_i = 1$ for all $i \in U$, we have: 
$v(\{ a , 2 , 3 , 4 \})=87$; 
$v(S) \leq 83$ if $a \in S$, $|S| = 4$ and $S \neq \{ a , 2 , 3 ,4 \}$;
$v(\{ a , 2 , 3 \})=84$;
$v(S) \leq 12$ if $a \in S$, $|S| = 3$ and $S \neq \{ a , 2 , 3 \}$;
$v(N) = 96$. 
Since $v(S) \leq v(N)$ for all $S \in \mathcal{P}^a$, we have $C(N,v) \neq \emptyset$. 
The exact value is $\eps^* = 4$,  the upper bound in~\eqref{e:c>0-ub} is 9/2, the lower bound~\eqref{e:c>0-lbmono} is 9/5, while $\eps^*_0=14/3$. 
In the same example, if it were $p_2=91$ then the upper bound in~\eqref{e:c>0-ub} would be equal to $\eps^* = 4$. 
The lower bound in~\eqref{e:c>0-lbmono} is tight if and only if the exact value is $\eps^* = 0$, or equivalently there exists a user $i$ such that $v(N) = v(N \setminus i)$. 
For instance, if it were $p_2=99$, then it would be $v(N) = v(\{ a , 2 , 3 , 4\}) = 96$.
\end{example}

\begin{remark}
Consider SESG when it is a clan game with $U_1=\{u_1\}$ (the case with $U_2=\{u_2\}$ is symmetric).
Hence, the monotonicity of the game requires that costs satisfy~\eqref{e:ex1-mono}.
Two cases occur: 
If $p_1 \geq \sum_{j \in U_2} q_j$, then $\eps^* >0$ and~\autoref{t:c>0eps-core} provides
$$
\eps^* =
\min\left\{ \frac{\sum_{j \in U_2} q_j - c_1}{|U_2|+2} , \ 
\min\left\{
\frac{\sum_{j \in S} q_j}{|S|+1}: \ 
S \subset U_2, \
S \neq \emptyset
\right\}
\right\}.
$$
If, in addition, $q_j=q$ for any $j \in U_2$, then
\begin{align}
\label{e:clan1-eps*}
    \eps^* = 
    \min\left\{ \frac{q|U_2|-c_1}{|U_2|+2} , \, \frac{q}{2} \right\}.
\end{align}
If $p_1 < \sum_{j \in U_2} q_j$, then either $\eps^*=0$ or $\eps^*>0$ can occur. 
Indeed,~\autoref{t:c>0eps-core} provides
$$
\eps^* = 
\min\left\{
\frac{p_1-c_1}{|U_2|+2} , \ 
\min\left\{
\frac{\left( p_1 -\sum_{j \in S} q_j \right)_+}{|U_2|+1-|S|}
: \ S \subset U_2, \ S \neq \emptyset
\right\}
\right\}.
$$
In addition, when $q_j=q$ for any $j \in U_2$, then $\eps^*=0$ if and only if $p_1/q \leq |U_2|-1$, otherwise
\begin{align}
\label{e:clan2-eps*}
\eps^*=
\begin{cases}
(p_1-c_1)/4
& \text{if $|U_2|=2$, $c_1 \geq 2q-p_1$},
\\
(p_1-c_1)/5
& \text{if $|U_2|=3$, $c_1 \geq 5q-3p_1/2$},
\\
[p_1-q(|U_2|-1)]/2
& \text{otherwise}.
\end{cases}
\end{align}
\end{remark}

\subsection{The least core in the unbalanced case}
\label{s:eps*unbal}

In this section we focus on the case where the admission fees are large, so that the core $C(N,v)$ is empty and hence $\eps^* < 0$ (see~\autoref{t:c>0core}) though $\eps^*_0 \geq 0$ always holds.
In this case the equivalences provided by~\autoref{t:c>0eps-core} a) do not hold since the allocation
$(\eps,\dots,\eps,v(N)-\eps|U|)$ never belongs to $C_\eps(N,v)$ whenever $\eps<0$.
Anyway, the least core value satisfies the inequality $\eps^* \leq
\hat{\eps}$
since the inequality $(iii)$ in~\autoref{t:c>0eps-core} a) holds whenever the $\eps$-strong core is nonempty regardless of the sign of $\eps$.
However, it is no longer necessarily an equality and indeed a better upper bound holds. 


\begin{proposition}
\label{prop:core-empty-ub}
Suppose $C(N,v) = \emptyset$. Then,
\begin{enumerate}[a)]
\item $2 \, \tilde{\eps} <
\eps^*
\leq 
\tilde{\eps}$ with $\tilde{\eps} := 
(v(N) - \max\{v(S):\ S \in \mathcal{P}^a \})/2 < 0$.

\item \inlineequation[e:leastcore-ub2]{
\eps^* \leq 
\bar{\eps} :=
\min \left\{ v(N)/|N| \,,\,
(v(N) - \max\{ \, v(S) \,:\, S \in \bar{\mathcal{P}}^a \, \})/2 \right\}.
}
\end{enumerate}
\end{proposition}

\begin{proof} 
a) Select any $x \in C_{\eps^*} (N,v)$, and $S \in \mathcal{P}^a$: we have $x(S) \geq v(S) + \eps^*$, and $x(N \setminus S) \geq v(N \setminus S) + \eps^* = \eps^*$ (since $a \notin N \setminus S$ and therefore $v(N \setminus S) = 0$). Hence
$$
 v(N) = x(N) = x(S) + x(N \setminus S) \geq v(S) + 2\eps^*.
$$
Thus, $\eps^* \leq ( \, v(N) - v(S) \, ) \,/\, 2$, and since this holds for any $S \in \mathcal{P}^a$ it holds in particular for the one giving the minimum, i.e., $\tilde{\eps}$. In order to prove that $\eps^* > 2\,\tilde{\eps}$, let $\delta \in ( \, 0 \,,\, -2\,\tilde{\eps} \,/\, |N| \, ]$, $\eps = 2\,\tilde{\eps} + \delta$, and $x = (-\delta,\ldots,-\delta,v(N) + \delta |U|)$. Then, $x \in C_\eps(N,v)$. Indeed, it is obvious that $x(N) = v(N)$. For $S \notin \mathcal{P}^a$ (thus, $v(S) = 0$) one has
%
%
$
 x(S) = -\delta|S| \geq -\delta|U| \geq 2\,\tilde{\eps} + \delta = v(S) + \eps
$
(the intermediate step hinges on the fact that $\delta|U| = \delta|N| - \delta$). 
If $S \in \mathcal{P}^a$, then
\begin{align*}
 x(S) = & \, x_a + x(S \cap U) = v(N) + \delta |U| - \delta |S \cap U| 
 \\
 \geq & \,
 v(N) + \delta |U| - \delta ( |U| - 1 ) =
 v(N) + \delta \geq v(S) + 2\,\tilde{\eps} + \delta = v(S) + \eps
\end{align*}
(the penultimate step uses $\tilde{\eps} \leq ( \, v(N) - v(S) \, ) \,/\, 2$). Therefore, we get  $\eps^* \geq \eps > 2\,\tilde{\eps}$.

\noindent b) The first term in the minimum comes from the inequality $(iii)$ in~\autoref{t:c>0eps-core} a)
with $S=\{a\}$, while the second from $\eps^* \leq \tilde{\eps}$ in a) since $\bar{\mathcal{P}}^a \subset \mathcal{P}^a$.
\end{proof}


Notice that $\tilde{\eps} \leq \hat{\eps} \leq \bar{\eps}$ always hold. 
Moreover, $\tilde{\eps}$ and $\hat{\eps}$ are both negative since there exists at least a set $S$ such that $v(S) > v(N)$ (see~\autoref{t:c>0core}), while $\bar{\eps}$ could be positive.
Anyway, the upper bound $\bar{\eps}$ is computationally affordable, while $\tilde{\eps}$ and $\hat{\eps}$ generally require the computation of an exponential number of coalition values.
Differently from the balanced case, 
a lower bound similar to~\eqref{e:c>0-lbmono} is not available because an exact formula for the least core value is missing and moreover monotonicity does not hold.

The following examples show that the upper bounds are tight, but do not always coincide with the least core value. 
Moreover, they could be different from each other.

\begin{example}
\label{ex:core-empty}
Consider SESG with $U_1=\{1,2\}$, $U_2=\{3\}$, $p_1=6$, $p_2=4$, $q_3=6$, $c_1 = c_3 = 0$ and $c_2=2$.
We have $
 v(\{ a , 1 , 2 \}) = -2
 \;\;,\;\; 
 v(\{ a , 1 , 3 \}) = 6 
 \;\;,\;\;
 v(\{ a , 2 , 3 \}) = 2$ and $v(N) = 4$. We have $\eps^* = \tilde{\eps}
=\hat{\eps} = \bar{\eps}=-1$, hence the upper bounds are tight.
\end{example}

\begin{example}
Consider SESG with $U_1=\{1,2,3\}$, $U_2=\{4,5,6\}$, $p_1=7$, $p_2=8$, $p_3=10$, $q_4=5$, $q_5=8$, $q_6=10$, $c_1=1$, $c_2=4$, $c_3=5$, $c_4=3$, $c_5=4$, $c_6=1$.
Computations provide $v(N)=5$, $\eps^* = -2/3$ and the upper bounds $\tilde{\eps}=-1/2$, $\hat{\eps}=-1/3$ and $\bar{\eps}=5/7$.
\end{example}



\subsection{The least core value via LP duality}
\label{sec:dual}

The least core value $\eps^*$ can be viewed also as the optimal value  of the linear program
\[
\textstyle
\max \big\{ \, \eps \,:\, 
x(S) \geq v(S) + \eps \;\; S \in \mathcal{P}
\;,\; x(N) = v(N)
\, \big \}
\]
so that linear programming theory can be exploited.
Since the equality constraint reads also $x_a = v(N) - x(U)$, then the variable $x_a$ can be eliminated and the inequality constraints recast accordingly.
Indeed, if $a \in S$, then the constraint reads
$
 v(S) + \eps \leq
 x(S)  = x(S \cap U) + x_a 
 = x(S \cap U) + v(N) - x(U)
 = v(N) - x(N \setminus S).
$
Notice that the constraint simply reads $x(S) \geq \eps$ when $a \notin S$.
Therefore, the computation of $\eps^*$ amounts to the linear program $(P_{LC})$:

\begin{subequations}
\begin{align}
 \max\ & \eps 
 \nonumber
 \\
 & \textstyle
   \eps - x(S) \leq 0
 & S \in \mathcal{P}^{-a} \label{eq:expf-e} 
 \\
 & \textstyle
   \eps + x(N \setminus S) \leq v(N) - v(S)
 & S \in \mathcal{P}^a \label{eq:expfe}
 \end{align}
 \label{eq:expP}
\end{subequations}
\noindent where $\mathcal{P}^{-a} = \mathcal{P} \setminus \mathcal{P}^a$. 
%
The dual problem $(D_{LC})$ thus has a variable $\lambda_S$ for each $S \in \mathcal{P}$:
\begin{subequations}
\begin{align}
\min \, & 
\textstyle \sum_{S \in \mathcal{P}^a} [v(N)-v(S)] \lambda_S & 
\nonumber
\\[2mm]
& 
\textstyle \sum_{S \in \mathcal{P}^{-a}_i} \lambda_S =
\sum_{S \in \mathcal{P}^a_{-i}} \lambda_S 
& i \in U \label{eq:expfD-i} 
\\[2mm]
& 
\textstyle \sum_{S \in \mathcal{P}} \lambda_S = 1
\label{eq:expfD-simplex} 
\\[2mm]
& \textstyle \lambda_S \geq 0
& S \in \mathcal{P} \label{eq:expfD-sign}
\end{align}
\label{eq:expD}
\end{subequations}
\noindent where $\mathcal{P}^{-a}_i = \{ \, S \in \mathcal{P}^{-a} \,:\, i \in S \, \}$ and  $\mathcal{P}^a_{-i} = \{ \, S \in \mathcal{P}^a \,:\, i \notin S \, \}$. 
Clearly the feasible region of $(P_{LC})$ is always nonempty and the optimal value $\eps^*$ is finite. Therefore, also the feasible region of $(D_{LC})$ is nonempty. Indeed, it is easy to show that $\lambda$ whose non-zero values are just  $\lambda_{\{i\}}=1/|N|$ for any $i \in N$ is feasible.

A primal feasible solution $(x,\eps)$ and a dual feasible solution $\lambda$ are optimal if and only if they satisfy the complementary slackness conditions, i.e.,
\begin{subequations}
\begin{align}
 & \lambda_S>0 \text{ and } S \in \mathcal{P}^{-a}
 \quad \Longrightarrow \quad x(S) = \eps, 
 \label{eq:expCS-e} 
 \\
 & \lambda_S>0 \text{ and } S \in \mathcal{P}^{a}
 \quad \Longrightarrow \quad v(N) - v(S) - x(N \setminus S) = \eps.
 \label{eq:expCSe}
\end{align}
\label{eq:expCF}
\end{subequations}
In the balanced case the least core value is already known by~\autoref{t:c>0eps-core} b).
Duality arguments allow  proving the same result, i.e., $\eps^*=\hat\eps$, through a different lens which will be useful to characterize those situations where it holds also in unbalanced games.

\noindent\textbf{Alternative proof of~\autoref{t:c>0eps-core} b).}
Consider the primal solution $(\hat{x},\hat\eps)$ with $\hat{x}=(\hat\eps,\dots, \hat\eps)$ and the dual solution $\hat\lambda$ with
\begin{equation}
\hat{\lambda}_S =
\begin{cases}
 1 \,/\, ( \, |N \setminus \hat{S}| + 1 \, )
 & \text{if } S = \hat{S} \text{ or } S = \{ \, i \, \} \text{ with } i \notin \hat{S},
 \\
 0 & \text{otherwise},
\end{cases}
\label{eq:barlambda}
\end{equation}
where $\hat{S} \in \mathcal{P}^a$ is any coalition providing $\hat\eps$, that is any $\hat{S}$ such that
$
\frac{v(N)-v(\hat{S})}{|N|-|\hat{S} \cap U|}
=
\hat\eps.
$
Notice that~\autoref{t:c>0core} guarantees $\hat\eps\geq 0$ since the game is balanced.

The above pair of solutions satisfy the complementarity slackness conditions by construction.
The primal solution is feasible since each~\eqref{eq:expf-e} reduces to $\hat\eps|S| \geq \hat\eps$ and the inequalities~\eqref{eq:expfe} collapse to the definition of $\hat\eps$.
The dual solution clearly satisfies~\eqref{eq:expfD-simplex} and~\eqref{eq:expfD-sign}. 
If $i \in \hat{S}$, then
$
 \sum_{S \in \mathcal{P}^{-a}_i} \hat{\lambda}_S 
 =
 \hat\lambda_{\{i\}}
 =
 0 =
 \sum_{S \in \mathcal{P}^a_{-i}} \hat{\lambda}_S
$
since $\hat{S} \notin \mathcal{P}^a_{-i}$ and $a \in \hat{S}$. 
If $i \notin \hat{S}$, then
$
 \sum_{S \in \mathcal{P}^{-a}_i} \hat{\lambda}_S 
 =
 \hat{\lambda}_{\{i\}}
 =
 \dfrac{1}{|N \setminus \hat{S}| + 1}
 =
 \hat{\lambda}_{\hat S}
 =
 \sum_{S \in \mathcal{P}^a_{-i}} \hat{\lambda}_S$.
Hence, $\hat\lambda$ satisfies~\eqref{eq:expfD-i} and therefore is feasible. 
As a consequence, the pair of solutions is optimal and the optimal value $\eps^*$ coincides with $\hat\eps$.
\qed

%
%
%

\

Notice that the primal optimal solution $(\hat{x},\hat{\eps})$ in the above proof is feasible for $(P_{LC})$ if and only if $\hat\eps \geq 0$. 
Therefore, the proof is not suitable for the unbalanced case.
Indeed, the same result $\eps^*=\hat\eps$ still holds in particular unbalanced cases where $\hat\eps$ is reached by ``leave-one-out'' coalitions, that is coalitions including all users but one. 


\begin{theorem}
\label{t:epsbar-epshat}
Suppose $C(N,v) = \emptyset$. 
Given any 
$
\hat{S} \in
\arg\min  
\left\{ 
\frac{v(N)-v(S)}{|N|-|S \cap U|}:\ S \in \mathcal{P}^a
\right\},
$
then $\eps^*=\hat{\eps}$ if and only if $|\hat{S}|=|N|-1$ and 
the unique $k \notin \hat{S}$ satisfies
\begin{equation}
\label{e:syst-epsbar-core-empty}
\left\{
\begin{array}{ll}
v(N) \geq v(S)+\hat{\eps} 
& \quad \forall\ S \in \mathcal{P}^{a}_k,
\\[3mm]
v(N) \geq v(S) + 2\,\hat{\eps} 
& \quad \forall\ S \in \mathcal{P}^{a}_{-k}.
\end{array}
\right.
\end{equation}
\end{theorem}



\begin{proof}
The feasible dual solution $\hat{\lambda}$ defined in~\eqref{eq:barlambda} yields 
the objective value 
$$
    \sum_{S \in \mathcal{P}^a} [v(N)-v(S)]\, \lambda_S
    =
    \dfrac{v(N)-v(\hat S)}{|N \setminus \hat S|+1}
    =
    \dfrac{v(N)-v(\hat S)}{|N| - |\hat S \cap U|}
    =
    \hat{\eps}    
$$
for the dual problem.
Therefore, $\eps^*=\hat{\eps}$ if and only if $\hat{\lambda}$ is an optimal, or equivalently there exists a feasible primal solution $(x,\eps)$ such that they satisfy the complementary slackness conditions~\eqref{eq:expCF}. 
This happens if and only if the linear system 
\begin{align*}
\begin{cases}
 \textstyle x(S) \geq \eps & \quad S \in \mathcal{P}^{-a} 
 \\
v(N) - v(S) \geq x(N \setminus S) + \eps & 
\quad S \in \mathcal{P}^a 
 \\
 x_i = \eps & \quad i \notin \hat{S}
 \\
 v(N)-v(\hat{S}) = x(N \setminus \hat{S}) + \eps &
\end{cases}
\end{align*}
has a solution $(x,\eps)$.
The above equalities imply $\eps=\hat{\eps}$, so that the system reduces to
\begin{align}
\label{e:syst-epsbar}
\begin{cases}
 x(S) \geq \hat\eps 
 & \quad S \in \mathcal{P}^{-a} 
 \\
v(N) - v(S) \geq x(N \setminus S) + \hat\eps 
& \quad S \in \mathcal{P}^a 
 \\
 x_i = \hat\eps 
 & \quad i \notin \hat{S}
\end{cases}
\end{align}
\emph{(Only if):}
Assume $|\hat{S}| \leq |N| - 2$,
namely there exist at least two different users $i, j \notin \hat{S}$, and take $x$ given by~\eqref{e:syst-epsbar}. 
The coalition $S=\{i,j\} \notin \mathcal{P}^{-a}$ provides the contradiction
$
 \hat{\eps} \leq x_i + x_{j} = 2\,\hat{\eps},
$
since $\hat\eps <0$.
Therefore, $|\hat{S}|=|N|-1$ must hold.
Given any $S \in \mathcal{P}^a_k$, then
\begin{align}
\label{e:x>=0LCeps<0}
x(N \setminus S) 
= x(\{k\} \cup (N \setminus S)) - x_k 
= x(\{k\} \cup (N \setminus S)) - \hat\eps
\geq 0
\end{align}
since $\{k\} \cup (N \setminus S) \in \mathcal{P}^{-a}$. Hence, 
$
 v(N) - v(S) \geq x(N \setminus S) + \hat{\eps}
 \geq \hat{\eps}.
$
Given any $S \in \mathcal{P}^a_{-k}$, then
$
v(N)-v(S) 
\geq x(N \setminus S) + \hat\eps
= x(N \setminus (S \cup \{k\})) + x_k + \hat\eps
\geq x_k + \hat\eps
= 2\hat\eps
$
since $k \in N \setminus S$ and $S \cup \{k\} \in \mathcal{P}^a_{k}$.

\noindent\emph{(If):} 
Choose $\check{x}$ given by
\begin{align}
\label{e:xcheck}
\check{x}_i =
\begin{cases}
\hat{\eps} & \text{if $i=k$},
\\
0 & \text{if $i \neq k$}.
\end{cases}
\end{align}
Clearly, $\check{x}$ satisfies the first and the last conditions in~\eqref{e:syst-epsbar} by its own definition, while the middle one collapses to~\eqref{e:syst-epsbar-core-empty}.
\end{proof}

Notice that~\eqref{e:syst-epsbar-core-empty} is equivalent to the feasibility of $\check{x}$ for $(P_{LC})$ which allows obtaining the result since $\check{x}$ and $\hat{\lambda}$ satisfy the complementarity slackness conditions.
The following example shows that the condition on the cardinality of $\hat{S}$ alone is not enough.

\begin{example}
\label{ex:core-empty3}
Consider SESG with $U_1=\{1,2,3\}$, $U_2=\{4,5,6\}$, $p_1=6$, $p_2=9$, $p_3=9$, $q_4=7$, $q_5=2$, $q_6=3$, $c_1=1$, $c_2=c_3=0$, $c_4=3$, $c_5=5$, $c_6=2$.
Computations provide $v(N)=1$, the least core value $\eps^* =\tilde{\eps}=-2$ and $\hat{\eps}=-3/2$ with $\hat{S}=N\setminus\{5\}$. Indeed, the coalition $\{a,2,3,4,6\}$ disproves~\eqref{e:syst-epsbar-core-empty} and therefore $\check{x}=(0,0,0,0,-3/2,0)$ is not feasible.
\end{example}


\autoref{t:epsbar-epshat} provides further insights on the tightness of the upper bounds of Section~\ref{s:eps*unbal}. 
Clearly, $\eps^*=\bar{\eps}$ guarantees $\eps^*=\hat{\eps}$. Unexpectedly the vice versa is also true in light of the condition on the cardinality of $\hat{S}$ in~\autoref{t:epsbar-epshat}.
On the contrary, the tightness of $\hat{\eps}$ obviously guarantees the tightness of $\tilde{\eps}$ but not vice versa, as shown by Example~\ref{ex:core-empty3}.


\section{Aggregator's shares}
\label{sec:aggshare}

In this section we aim at evaluating the maximum and minimum shares, $M_a$ and $m_a$, that the aggregator can obtain by the allocations in the least core $C_{\eps^*}(N,v)$ when $\eps^*=\hat{\eps}$.

Actually, the value of the maximum share $M_a$ follows immediately from the previous section. 
If the game is balanced, then $M_a = v(N) - \eps^*|U|$ since each user must receive at least $\eps^*$ and $( \, \eps^* \,, \ldots \,,\, \eps^* \,,\, v(N) - \eps^*|U| \, ) \in C_{\eps^*}(N,v)$ by~\autoref{t:c>0eps-core}.
If the game is not balanced and $\eps^*=\hat{\eps}$, then $M_a=v(N)-\eps^*$ 
since any allocation $x$ in the least core satisfies $x(U) \geq \eps^*$ (see~\eqref{eq:expf-e} and \eqref{e:x>=0LCeps<0}) and $\check{x} \in C_{\eps^*}(N,v)$, with $\check{x}_k = \eps^*$ and $\check{x}_i=0$ for all $i \neq k$, by~\autoref{t:epsbar-epshat}.

On the contrary, the minimum share $m_a$ depends upon the game at hand.
Clearly, it always holds $m_a = v(N) - M_U$, where $M_U$ is the maximum overall share that users can obtain. 
The computation of this latter quantity amounts to the linear program $(P^*_U)$
\begin{subequations}
\begin{align}
 \max\ \ & x(U)
 \nonumber
 \\
 & x(S) \geq \eps^* && S \in \mathcal{P}^{-a}
   \label{eq:PA-e}
\\
 & x(N \setminus S) \leq v(N) - v(S) - \eps^*
 && S \in \mathcal{P}^a \label{eq:PAe}
\end{align}
\label{eq:PA}
\end{subequations}
whose dual problem $(D^*_U)$ reads
\begin{subequations}
\begin{align}
 \min\ & 
\textstyle \sum_{S \in \mathcal{P}^a} \left[ v(N) - v(S) - \eps^* \right] \lambda_S 
 - \eps^* \sum _{S \in \mathcal{P}^{-a}} \lambda_S 
 \nonumber
 \\[2mm]
 & \textstyle\sum_{S \in \mathcal{P}^a_{-i}} \lambda_S -
   \sum_{S \in \mathcal{P}^{-a}_i} \lambda_S = 1
 & i \in U \label{eq:DAi} 
 \\
 & \textstyle\lambda_S \geq 0 & S \in \mathcal{P}. 
 \label{eq:DAsign}
\end{align}
\label{eq:DA}
\end{subequations}
Notice that the constraints~\eqref{eq:PA} describe the least core (with $x_a$ implicitly given). 
Actually, some variables may happen to be fixed to the value $\eps^*$. 
Indeed, this always happens for those users that do not belong to all the coalitions providing $\hat{\eps}$.

\begin{proposition}
\label{prop:bloc}
Assume $\eps^*=\hat{\eps}$ and let 
$
 \mathcal{R} = \arg\min
 \left\{ 
  \dfrac{v(N) - v(S)}{|N| - |S \cap U|} \,:\, S \in \mathcal{P}^a
 \right\}$ and $S_{\min} = \bigcap_{S \in \mathcal{R}} S$. 
 Then, any $x \in C_{\eps^*}(N,v)$ satisfies $x_i = \eps^*$ for any $i \notin S_{\min}$.
\end{proposition}

\begin{proof}
Let $S \in \mathcal{R}$ be arbitrary. Then, the following chain of inequalities holds:
\begin{align*}
  x(N \setminus S)
 \leq v(N)-v(S)-\eps^* 
  = |N \setminus S| \, (v(N)-v(S))/(|N \setminus S|+1)
   = \eps^*|N \setminus S|,
\end{align*}
where the inequality follow from the constraints~\eqref{eq:PAe} and the equalities follow by the assumption $\eps^*=\hat{\eps}$. 
Since~\eqref{eq:PA-e} guarantee $x_i \geq \eps^*$ for any $i \in U$, we get $x_i = \eps^*$ for all $i \in N \setminus S$.    
\end{proof}

\begin{corollary}
\label{r:checkx}
If $C(N,v)=\emptyset$ and $\eps^*=\hat{\eps}$, then
$|\mathcal{R}|=1$.
\end{corollary}

\begin{proof}
If $|\mathcal{R}| \geq 2$, then $N \setminus\{k_1\}, N \setminus\{k_2\} \in \mathcal{R}$ for some $k_1, k_2 \in U$ by~\autoref{t:epsbar-epshat}.
Therefore, any $x \in C_{\eps^*}(N,v)$ satisfies $x_{k_1}=x_{k_2}=\eps^*$ that contradicts the constraint~\eqref{eq:PA-e}
as $\eps^*<0$.
\end{proof}

\begin{remark}
\label{rem:bbma}
Consider SESG in the case of the big-boss game considered in Proposition~\ref{prop:ex1-bbg-clan} c) (the case d) is analogous).
By Remark~\ref{rem:bbg-clan-c0}, $\eps^*=0$ and moreover $S_{\min}=U_2 \cup \{a\}$ since $M^v(i)=0$ holds if and only if $i \in U_1$.
Proposition~\ref{prop:bloc} implies that any $x \in C_{\eps^*}(N,v)$ satisfies $x_i=0$ for any $i \in U_1$.
Hence, the feasible region of $(P^*_U)$
reduces to the vectors $x$ with the above zero values and $x_i \in [0, q_i]$ for any $i \in U_2$, in accordance also with~\cite[Theorem 3.2]{MuNaPoTi88}.
Therefore, we have $m_a=0$ and $M_a=v(N)$.
\end{remark}

\begin{remark}
Consider SESG when it is a clan game with $U_1=\{u_1\}$ and $q_j=q$ for any $j \in U_2$. 
Two cases occur: if $p_1 \geq q|U_2|$, then $S_{\min} =\{a\}$ if $|U_2|=2$, or $|U_2|=3$ and $c_1 \geq q/2$, 
or $|U_2|=4$ and $c_1=q$; $S_{\min} =\{a,u_1\}$ otherwise.
If $p_1 < q|U_2|$, then $S_{\min} = \{a\}$ if $c_1 = p_1 \leq q(|U_2|-1)$,
or $|U_2|=2$, $p_1>q(|U_2|-1)$ and $c_1 \geq 2q-p_1$,
or $|U_2|=3$, $p_1>q(|U_2|-1)$ and $c_1 \geq 5q - 3p_1/2$; 
$S_{\min} = \{a,u_1\}$ otherwise.
No matter the value of the ratio $p_1/q$, whenever $S_{\min}=\{a\}$ Proposition~\ref{prop:bloc} guarantees  $x_i=\eps^*$ for any $i \in U$.
Since $\eps^*$ is given by \eqref{e:clan1-eps*} or \eqref{e:clan2-eps*} according to the different settings, it always holds $\eps^*=v(N)/|N|$ and hence $m_a=M_a=\eps^*$. 

Whenever $S_{\min}=\{a,u_1\}$, Proposition~\ref{prop:bloc} guarantees that the least core reduces to
$$
C_{\eps^*}(N,v)=
\left\{
x \in \R^{|N|}:\ 
x_i=\eps^*, \  i \in U_2, \quad 
x_1 \geq \eps^*, \quad
x_a \geq \eps^*, \quad 
x_1+x_a = v(N)-\eps^*|U_2|
\right\},
$$
thus $m_a=\eps^*$, while $M_a>m_a$ since \eqref{e:clan1-eps*}--\eqref{e:clan2-eps*} guarantee $\eps^* < v(N)/|N|$ in these situations.
\end{remark}

Beyond the above clan games, there are more general situations where it is not possible to give the aggregator less than the maximum, that is $M_a = m_a$.
In the balanced case it happens if and only if $S_{\min}$ reduces to the aggregator only. Indeed, the following characterizations of this situation hold.

\begin{theorem}
\label{teo:maxaggr}
Suppose $C(N,v) \neq \emptyset$. The following statements are equivalent:

\begin{enumerate}[(i)]
 \item $m_a = M_a$ 
 
 \item the optimal solution of $(P^*_U)$ is $x^*=(\eps^*,\dots,\eps^*)$
 
 \item $C_{\eps^*}(N,v)=\{(x^*,M_a)\}$
 
 \item $S_{\min} = \{ \, a \, \}$.
\end{enumerate}
\end{theorem}

\begin{proof}
\noindent The equivalence between (i), (ii) and (iii) is straightforward since~\autoref{t:c>0eps-core} a) guarantees that $(x^*,M_a) \in C_{\eps^*}(N,v)$ always holds.
Moreover, (iv) implies all of them thanks to Proposition~\ref{prop:bloc}. 
Vice versa, assume that $x^*$ is the optimal solution of $(P^*_U)$.
The active constraints at $x^*$ correspond to the singletons $\{i\}$ for all $i \in U$ and all the sets $S \in \mathcal{R}$ thanks to~\autoref{t:c>0eps-core} b). 
Consider an optimal solution $\lambda^*$ of $(D^*_U)$:
the constraints~\eqref{eq:DAi} read
$
 \sum_{S \in \mathcal{R} \,:\, i \notin S} \lambda^*_S 
 = 
 \lambda^*_{\{ \, i \, \}} + 1$ for any $i \in U$, 
 in light of complementary slackness. 
Therefore, the feasibility of $\lambda^*$ guarantees that the left-hand side of the above equation is positive.
Hence for any $i \in U$ there must exist $S \in \mathcal{R}$ such that $i \notin S$, which implies $S_{\min}=\{a\}$.
\end{proof}

In the unbalanced case the characterization of the equality $m_a=M_a$ is more involved. 
Notice that $S_{\min}$ never reduces to the singleton $\{a\}$ and actually it includes all users but one, namely $S_{\min}=\hat{S}=N \setminus\{k\}$ for some $k \in U$, according to~\autoref{t:epsbar-epshat} and Corollary~\ref{r:checkx}.
Indeed, the role of $S_{\min}$ in~\autoref{teo:maxaggr} is now played by the intersection of a suitable family of coalitions.

\begin{theorem}
\label{teo:maxaggr-empty}
Suppose $C(N,v)=\emptyset$ and $\eps^*=\hat{\eps}$. The following statements are equivalent:
\begin{enumerate}[(i)]
 \item $m_a = M_a$
 
 \item the optimal solution of $(P^*_U)$ is 
$\check{x}$ defined as
$\check{x}_i =
\begin{cases}
\hat{\eps} & \text{if $i=k$},
\\
0 & \text{if $i \neq k$}.
\end{cases}$

 \item $C_{\eps^*}(N,v) = \{(\check{x},M_a)\}$
 
 \item there exists a family of coalitions $\{S_l\}_{l \in I} \subset \mathcal{P}^a_k$ such that
 \begin{enumerate}[a)]
 \item $\bigcap\limits_{l \in I} S_l = \{a,k\}$,

 \item $v(S_l)=v(N)-\eps^*$ for any $l \in I$,
 \end{enumerate}
 where $I$ is an index set with $|I| \leq |U|-1$.

\end{enumerate}
\end{theorem}






\begin{proof}
\noindent The equivalence between (i) and (ii) is straightforward since $(\check{x},M_a) \in C_{\eps^*}(N,v)$ thanks to~\autoref{t:epsbar-epshat}, while the equivalence between (ii) and (iii) follows from 
Remark~\ref{r:checkx}.

Assume (iv) holds and take any $x$ feasible for $(P^*_U)$. 
The constraint~\eqref{eq:PAe} for any coalition $S_l$ of the family reads
$x(N \setminus S_l) \leq~0$. Moreover, Remark~\ref{r:checkx} guarantees $x(N \setminus S_l) \geq 0$ as $k \notin S_l$ and therefore $x_i=0$ for any $i \notin S_l$. 
As a consequence of condition a) $x_k$ is the only non-zero component and takes value $\eps^*$ by Proposition~\ref{prop:bloc}. Hence, $x=\check{x}$ and (iii) follows.

Vice versa, assume (iii), so that $\check{x}$ is the optimal solution of $(P^*_U)$.
The active constraints at $\check{x}$ correspond to coalitions $S \in \mathcal{P}^{-a}_k$, $S=\hat{S}$ and $S \in \mathcal{P}^*$, where 
$
\mathcal{P}^*
:=\{ S \in  \mathcal{P}^a_k:\ v(N)-v(S)=\eps^* \}.
$
Consider an optimal solution $\check{\lambda}$ of $(D^*_U)$:
the constraints~\eqref{eq:DAi} read
$
 \sum_{S \in \mathcal{P}^a_{-i} \cap \mathcal{P}^*} \check\lambda_S 
 = 
 \sum_{S \in \mathcal{P}^{-a}_{i} \cap \mathcal{P}^{-a}_k} \check\lambda_S 
  + 1$ for any $i \in U\setminus\{k\}$
in light of complementary slackness. 
Therefore, the feasibility of  $\check\lambda$ guarantees that the left-hand side of the above equation is positive.
Hence, for any $i \in U\setminus\{k\}$ there must exist $S(i) \in \mathcal{P}^a_{-i} \cap \mathcal{P}^*$.
Therefore, the family $\{S(i)\}_{i \neq k}$ satisfies conditions a) and b). 
\end{proof}

A lower bound for $m_a$ can be obtained by exploiting any partition of the set of users since the corresponding constraints~\eqref{eq:PAe} provide an upper bound on the maximum overall share of the users.

\begin{theorem}
\label{teo:aggr-part}
Let $\mathcal{P}(U)$ be the set of all partitions of the set $U$. Then,
\begin{enumerate}[a)]
 \item \inlineequation[eq:aggr-part]{m_a \geq v(N) - 
       \min \left\{ \, |\mathcal{A}| (v(N)-\eps^*) - \sum\limits_{S \in \mathcal{A}} v(N \setminus S) 
       \,:\, \mathcal{A} \in \mathcal{P}(U) \, \right\},
       }
 
\item equality holds in~\eqref{eq:aggr-part} if and only if there exists an optimal solution $\lambda^*$ of $(D^*_U)$ such that $\mathcal{A^*} = \{ \, S \,:\, \lambda^*_{N \setminus S} > 0 \, \} \in \mathcal{P}(U)$.
\end{enumerate}
\end{theorem}

\begin{proof}
a) Let $\mathcal{A} \in \mathcal{P}(U)$. Then, any feasible $x$ for $(P^*_U)$ satisfies 
$$
\textstyle
x(U) = \sum_{S \in \mathcal{A}} x(S) \leq 
\sum_{S \in \mathcal{A}} \left( v(N) - v(N \setminus S) - \eps^*\right) =
|\mathcal{A}| (v(N) - \eps^*) - \sum_{S \in \mathcal{A}} v(N \setminus S),
$$
where the inequality follows from~\eqref{eq:PAe}. Therefore, 
$
M_U \leq 
\min \big\{ \, |\mathcal{A}| (v(N)-\eps^*) - \sum_{S \in \mathcal{A}} v(N \setminus S) 
       \,:\, \mathcal{A} \in \mathcal{P}(U) \, \big\}$
and the thesis follows since $m_a = v(N) - M_U$.

\noindent b) \emph{(If):} Since $\mathcal{A^*}$ is a partition of $U$, $a \in N \setminus S$ for any set $S$ such that $\lambda^*_{N \setminus S} > 0$. 
For any optimal solution $x^*$ of $(P^*_U)$, the complementary slackness conditions require that the corresponding constraint~\eqref{eq:PAe} is satisfied as an equality, i.e., $x^*(S) = v(N) - v(N \setminus S) - \eps^*$. Therefore, we have
$
M_U = x^*(U) 
= \sum_{S \in \mathcal{A^*}} x^*(S)
= \sum_{S \in \mathcal{A^*}} v(N) - v(N \setminus S) - \eps^*
= |\mathcal{A^*}| (v(N)-\eps^*) - \sum_{S \in \mathcal{A^*}} v(N \setminus S)
\,
$
i.e., \eqref{eq:aggr-part} holds as an equality.

\emph{(Only if):} Let $\mathcal{A^*}$ be a partition of $U$ such that~\eqref{eq:aggr-part} holds as equality. 
Hence, there exists an optimal solution $x^*$ of $(P^*_U)$ such that 
$|\mathcal{A^*}| (v(N)-\eps^*) - \sum_{S \in \mathcal{A^*}} v(N \setminus S) 
= x^*(U) 
= \sum_{S \in \mathcal{A^*}} x^*(S)$.
Since $S \in \mathcal{A^*}$ implies $S \in \mathcal{P}^{-a}$, \eqref{eq:PAe} guarantees that $x^*(S) \leq  v(N) - v(N \setminus S) - \eps^*$ holds for any $S \in \mathcal{A^*}$.
Therefore, the above chain of equalities implies that $x^*(S) = v(N) - v(N \setminus S) - \eps^*$ for all $S \in \mathcal{A^*}$. 
Then, the dual solution $\lambda^*$ defined as follows: $\lambda^*_{S} = 1$ if $N \setminus S \in \mathcal{A^*}$ and $\lambda^*_{S} = 0$ otherwise, satisfies the complementarity slackness together with $x^*$ by construction. Notice that $\lambda^*_S = 0$ for all $S \in \mathcal{P}^{-a}$, and for any $i \in U$ there exists a unique set $S_i \in \mathcal{P}^a$ such that $N \setminus S_i \in \mathcal{A^*}$ and $i \in N \setminus S_i$. 
Hence, for all $i \in U$
$
\sum_{S \in \mathcal{P}^a_{-i}} \lambda^*_S -
\sum_{S \in \mathcal{P}^{-a}_i} \lambda^*_S =
\sum_{S \in \mathcal{P}^a_{-i}} \lambda^*_S =
\lambda^*_{S_i} = 1$, i.e., $\lambda^*$ is feasible for $(D^*_U)$ and hence it is optimal.
\end{proof}

Clearly, the above lower bound is computationally unaffordable. 
Any partition of $U$ provides a weaker bound that on the contrary can be easily computed. 
In particular, the partition made by the singletons provides a lower bound based on the marginal contributions, namely
\begin{align}
\label{e:lbsingle}
m_a \geq v(N)+\eps^*|U|-\sum_{i \in U} M^v(i).
\end{align}
The bound is always tight for any simple ESG that is a big-boss game. 
Indeed, both $m_a=0$ and $\eps^*=0$ hold (see Remarks~\ref{rem:bbma}) and $v(N)$ is the sum of the marginal contributions (see Proposition~\ref{prop:ex1-bbg-clan} and Remark~\ref{rem:bbg-clan-c0}) in such a case.
Notice that the tightness of the bound implies that the partition provides the minimum in~\eqref{eq:aggr-part} as well.
The bound~\eqref{e:lbsingle} could be tight also for other games as the following example shows.

\begin{example}
Consider SESG with $U_1=\{1,2\}$, $U_2=\{3,4\}$, $p_1=9$, $p_2=8$, $q_3=2$, $q_4=4$, $c_1=c_2=c_3=0$, $c_4 = 1$.
Proposition~\ref{prop:ex1-mono} implies that this is not a big-boss game.
Since $v(N)=5$, $\eps^* = m_a = 0$, $M^v(1)=M^v(2)=0$, $M^v(3)=2$, $M^v(4)=3$, then~\eqref{e:lbsingle} holds as an equality.
\end{example}

Nonetheless, the inequality in~\eqref{eq:aggr-part} can be strict both in the balanced and unbalanced case as the following examples show.

\begin{example}
\label{ex:lb-ma-strict1}
Consider SESG with $U_1=\{1,2\}$, $U_2=\{3,4\}$, 
$p_1=p_2=6$, $q_3=5$, $q_4=10$, 
$c_1=c_2=0$, $c_3=2$, $c_4=7$.
Computations show that $v(N)=3$, $\eps^* =0$, $m_a = 3$, 
while the right hand side in~\eqref{eq:aggr-part} is equal to 0, since the minimum over all possible partitions of $U$ is 3.

\end{example}

\begin{example}
\label{ex:lb-ma-strict2}
Consider SESG with $U_1=\{1,2\}$, $U_2=\{3,4\}$, $p_1=p_2=6$, $q_3=5$, $q_4=10$, $c_1=c_2 = 0$, $c_3 = 4$, $c_4 = 8$.
Computations show that $v(N)=0$, $\eps^* = \hat{\eps}=-1$, $m_a = 1$, 
while the right hand side in~\eqref{eq:aggr-part} equals  0, since the minimum over all possible partitions of $U$ is 0.

\end{example}



\section{Algorithmic approaches to the computation of the least core value}
\label{sec:algo}

When ESG is balanced 
or the conditions in \autoref{t:epsbar-epshat} hold, the least core value $\eps^*$ is provided by the formula
\begin{equation}\label{e:eps*}
\min \left\{ 
\frac{v(N)-v(S)}{|N|-|S \cap U|}:\ 
S \in \mathcal{P}^a
\right\}
\end{equation}
from~\autoref{t:c>0eps-core} or equivalently the optimal value of problem~\eqref{eq:expP}, that has exponentially many constraints, namely one for each possible coalition.
Their na\"ive implementation would require the solution of $2^{|N|-1}$ optimization problems~\eqref{eq:venabler}, i.e., the computation of the values $v(S)$ for all $S \in \mathcal{P}^a$ and $S=N$, which is not doable in practice even for a medium size community. 
In this section, we also explore two different approaches for the computation of $\eps^*$ that appear much more efficient in practice, although both still have worst-case exponential complexity.

While the computation of $v(N)$ is necessary, the explicit computation of some values $v(S)$ might be avoided. 
Indeed, the coalitions with the same number $k$ of users can be considered all together since~\eqref{e:eps*} can be recast as
\begin{equation}\label{e:minratio}
\min \left\{ \frac{v(N) - w_k}{|N|-k}: 
\ k = 0, \dots, |U| - 1
\right\},
\end{equation}
where $w_k=\max \{ v(S): \ S \in \mathcal{P}^a, \  |S \cap U| = k\}$.
If $k=0$ or $k=1$, then $w_k=0$ according to~\eqref{eq:vSdef}. 
Otherwise, it can be computed by introducing ``design variables'' $s_i \in \{0,1\}$ that indicate whether or not user $i \in U$ belongs to the coalition. 
Therefore, $w_k$ is the optimal value of the following optimization problem that is built  encompassing all the optimization problems~\eqref{eq:venabler} with $k$ users into a unique one:
\begin{subequations}
\begin{align}
\max\ & 
\sum_{i \in U}[ f_i(z_i,y_i) - c_i s_i - b_i ] +
f_a \left( \sum_{i \in U} y^+_i \right) &
\label{eq:venabler-of} 
\\
 & (z_i, y_i) \in D_i & i \in U
\label{eq:venabler-D} 
\\
& \sum_{i \in U} y^+_i = \sum_{i \in U} y^-_i
\label{e:balance2}
\\
& \sum_{i \in U} s_i = k
\label{eq:eps*compactk-k} 
\\
  & 0 \leq y_i \leq  s_i \, \bar{y}_i 
  & i \in U
	\label{eq:eps*compact1-semicont} 
\\
& s_i \in \{0,1\} & i \in U
    \label{eq:eps*compact1-zbin}
\end{align}
\label{e:form-k}
\end{subequations}
where $\bar{y}_i \in \R^{2m}_+$ is any upper bound for the projection of the bounded $D_i$ on the $y_i$ space.
 
Actually, the computation of the least core value can be performed through a unique optimization problem. 
Indeed, the coalition with 0 or a single user entail the upper bound $v(N)/|N|$ and \eqref{e:minratio} can be considered only for $k \geq 2$.
This latter quantity, i.e.,
$$
\eta^a_+ 
:=
\min \left\{ \frac{v(N) - w_k}{|N|-k} : \ k = 2, \dots, |U| - 1
\right\}
=
\min \left\{ \frac{v(N) - v(S)}{|N|-|S \cap U|} : \ S \in \mathcal{P}^a_+
\right\},
$$ 
amounts to the optimal value of the following program
\begin{subequations}
 \begin{align}
  \min\ & \eps \label{eq:eps*compact1-obj} 
  \\
  & \eps \left( |N| - \sum_{i \in U} s_i \right) 
  \geq 
  v(N) - \sum_{i \in U} [ f_i(z_i,y_i) - c_i s_i - b_i ] - 
  f_a \left( \sum_{i \in U} y^+_i \right)
  \label{eq:eps*compact1-nonlin} 
  \\
  & (z_i, y_i) \in D_i & i \in U
    \label{eq:eps*compact1-ind} 
    \\
& \sum_{i \in U} y^+_i = \sum_{i \in U} y^-_i
\label{e:balance3}
\\
& 2 \leq \sum_{i \in U} s_i \leq |U|-1
\label{eq:eps*compactk-2} 
\\
  & 0 \leq y_i \leq  s_i \, \bar{y}_i 
  & i \in U
	\label{eq:eps*compact1-semicont2} 
\\
& s_i \in \{0,1\} & i \in U
    \label{eq:eps*compact1-zbin2}
 \end{align}
 \label{e:form-compact}
\end{subequations}
Indeed, when the variables $s_i$ have been fixed, a unique coalition $S$ is at hand and the constraint~\eqref{eq:eps*compact1-nonlin} actually becomes  
\[
 \eps \geq (v(N) - v(S)) / (|N| - |S \cap U|)
\]
since any optimal solution of~\eqref{eq:venabler} provides the minimum value for the right-hand side of~\eqref{eq:eps*compact1-nonlin}. 
To achieve optimality, $S$ will be chosen as the coalition that minimizes the right-hand side in the above formula.

If $D_i$ is a polyhedron and $f_i$ and $f_a$ are linear, then each problem~\eqref{eq:venabler} is a Linear Program, while each problem~\eqref{e:form-k} is a Mixed-Integer Linear Program. 
Anyway, the computation of $\eps^*$ through~\eqref{e:form-k} requires the resolution of only $|U|-2$ MILPs.
On the contrary, the computation through~\eqref{e:form-compact} calls for a unique mixed-integer problem, but the constraint~\eqref{eq:eps*compact1-nonlin} is quadratic and non-convex so that~\eqref{e:form-compact} is a Mixed-Integer Quadratically Constrained Linear Program.
In this framework, problems~\eqref{eq:expP}, \eqref{e:form-k} and \eqref{e:form-compact}
can be efficiently solved by handling state-of-the-art general-purpose optimization solvers in suitable ways.

Since it has an exponential number of constraints, problem~\eqref{eq:expP} can be solved by a \emph{row generation} technique following the approach developed in~\cite{FiBiFrPaPo25} in a more general context.
Since ESG is balanced, the constraints~\eqref{eq:expf-e} reduce to only $x_i \geq \eps$ for all $i \in U$ as the others become redundant. 
On the contrary, the constraints~\eqref{eq:expfe} are   considered only for a small subset of $\mathcal{P}^a$.
The optimal solution $(x^\diamond, \eps^\diamond)$ of the resulting \emph{master problem} is analyzed. 
If it is not feasible for the overall problem~\eqref{eq:expP}, then the master problem is updated adding the most violated constraint at $(x^\diamond, \eps^\diamond)$.
This can be done by solving
\begin{equation}
 \max \{ v(S) + x^\diamond(N \setminus S)
 : \ S \in \mathcal{P}^a \}.
 \label{eq:separation}    
\end{equation}
Indeed, if the optimal value is less or equal to $v(N)-\eps^\diamond$, then $\eps^*=\eps^\diamond$. 
Otherwise, the constraint~\eqref{eq:expfe} corresponding to an optimal coalition is added to the master problem and the process is iterated.
Problem~\eqref{eq:separation} can be formulated as the mixed-integer program obtained from~\eqref{e:form-k} by adding the scalar product $(1-s)^\top x^\diamond$
to the objective function and replacing constraint~\eqref{eq:eps*compactk-k}  by~\eqref{eq:eps*compactk-2}.

Problems~\eqref{e:form-k} can be solved iteratively for $k=2,\dots,|U|-1$ independently of each other. 
However, some techniques could be used to speed-up their resolution.
Indeed, an initial upper bound $\eps^\uparrow$ for the least core value is available, namely $v(N)/|N|$ since $w_0=w_1=0$, and it can be improved when the resolution of~\eqref{e:form-k} for a given $k$ provides a better value, i.e.,  when $v(N) - \eps^\uparrow ( |N| - k ) < w_k$ holds.
This allows to put an \emph{upper threshold} on the objective function. 
In fact, all algorithms for (exactly) solving problems like~\eqref{e:form-k} routinely compute upper bounds $\bar{w}_k \geq w_k$ that are progressively revised until they converge to the optimal value. 
Whenever a bound is found such that $v(N) - \eps^\uparrow (|N| - k) \geq \bar{w}_k$ holds, the resolution can be stopped.
Clearly, the order of resolution of the problems may play a significant role in the potential speed-up and the increasing and decreasing order for $k$ are the most natural choices.

Overall, the reformulations of the least core value through the above optimization problems lead to five exact procedures: the row generation approach to solve~\eqref{eq:expP} (shortly RG), the resolution of the problems~\eqref{e:form-k} iteratively for all $k$ in no particular order (IT), with early termination via upper thresholds in increasing and decreasing order (ITI and ITD, respectively), and the resolution of the compact formulation~\eqref{e:form-compact} with a quadratic constraint (CQC).

To assess their performances, the approaches have been applied to a set of synthetic instances with $|U| = 10$, 20, 50, 100 and 200 users. 
No admission fees have been set so that ESG is always balanced. 
All instances have been generated from a common set of 10 representative users as in the base configuration in~\cite{FiFP21}: 3 users are pure consumers, while the other 7 are prosumers. 
Larger communities are constructed by replicating this composition so as to preserve comparable aggregate characteristics across sizes, similarly to~\cite{FiBiFrPaPo25}.
Also the objective function and constraints are taken from~\cite{FiBiFrPaPo25}, hence the optimization problems~\eqref{eq:venabler} are linear.  

The code has been written in the {\tt Julia} language, exploiting the {\tt JuMP} modelling system and in particular the {\tt EnergyCommunity.jl} and the {\tt TheoryOfGames.jl} packages in~\cite{EC301,TG200}. 
All the mixed-integer programs have been solved using the latest {\tt Gurobi} solver, version 13.0, using 8 threads on a four 18-core Intel Xeon Gold 6140M 2.30GHz machine with 1.2TB RAM running Centos Linux 7, with default parameters settings. 
The total running time of each approach is reported in Table~\ref{tab:res} for the different sizes of the energy communities.
In the last line, the table reports also the time for computing $v(N)$ and $w_{(|U|-1)}$ only, that together give the upper bound~\eqref{e:leastcore-ub}. 
Indeed, the computation of this upper bound can be viewed as an heuristic approach for the least core value since it relies only on the largest coalitions (LCH). 

\begin{table}[tb]
	\small
\centering
\begin{tabular}{l@{\qquad} ccccc}
 & \multicolumn{5}{c}{$|U|$}\\
 \cmidrule(){2-6}
&  10 &   20 &   50 &   100 &   200 \\
\toprule
RG    & 173 &  474 & 1977 &  8825 & 32690 \\
IT   & 159 & 1972 & 5976 & 33558 &   --    \\
ITI  & 155 & 1604 & 4694 & 19639 &   --    \\
ITD  & 131 &  846 & 1400 &  6781 & 41533 \\
CQC   &  66 &  123 &  140 &   292 &   643 \\
LCH &  43 &  114 &  148 &   264 &   583 \\
\bottomrule
\end{tabular}
\caption{Running time in seconds for the solution approaches.}
\label{tab:res}
\end{table}

The approaches IT and ITI have proved to be the least efficient by far and have been ran up to $|U| = 100$ only since their (likely prohibitive) resolution for $|U| = 200$ would not have provided valuable new insights. 
In all cases, all the approaches computed the least core value with at least the same first six decimal digits, that is coherent with the standard {\tt 1e-6} relative stopping accuracy of the solver.
Remarkably, this is true even for the heuristic approach LCH, showing that the ``leave-one-out'' coalitions are likely to yield the least core value.
As further evidence, ITD is the unique approach based on problem~\eqref{e:form-k} with an acceptable performance, probably because larger coalitions are considered first in setting the upper thresholds. 
Also, it is generally competitive with the previous RG approach. 
Moreover, both approaches exhibit a quadratic-like growth with the number of users, which is still relatively good considering that a number of $\mathcal{NP}$-hard problems of increasing size are solved.

CQC is by far the most efficient exact approach and it is competitive with the LCH heuristic which is consistently the fastest as expected.
The running time of CQC shows a linear-like growth with $|U|$, therefore making it very likely to efficiently compute the least core value for energy communities with about one thousand users.




\section{Conclusions}
\label{sec:conc}

We modeled a quite general energy community as a cooperative game, where the worth of a coalition is the optimal value of a suitable optimization problem.
In order to identify fair and stable allocations of the overall reward, we focused on the least core of the game. 
In particular, an exact formula and computable bounds for the least core value have been obtained. 
The bounds require the computation of the worth of few coalitions only, precisely those with all prosumers but one.
While in principle the exact formula requires the computation of the worth of all possible coalitions, mixed-integer programs have been exploited to reduce the computational burden that appear to be comparable with that of the bounds.

These results could be exploited for the selection of a specific fair and stable allocation according to some further criterion, for instance by optimizing it over the least core.
Indeed, some computational tests in this direction have already been performed in~\cite{FiBiFrPaPo25}, where the least core value has been computed by the RG technique.
Therefore, the exploitation of the CQC or LHC approaches from the last section would provide a meaningful computational boost.



\section*{Acknowledgements}

\noindent G. Bigi, A. Frangioni, and M. Passacantando are members of  Gruppo Nazionale per l’Analisi Matematica, la Probabilità e le loro Applicazioni (GNAMPA - National Group for Mathematical Analysis, Probability and their Applications) of Istituto Nazionale di Alta Matematica (INdAM - National Institute of Higher Mathematics).
The authors, to various extends, gratefully acknowledge the financial contribution from the European Union -- Next-GenerationEU -- National Recovery and Resilience Plan (NRRP) –- Mission 4, Component 2, Investment n.~1.1 (call PRIN 2022 D.D. 104 02-02-2022, project title ``Large-scale optimization for sustainable and resilient energy systems'', CUP I53D23002310006) and Investment 1.3 (call for tender No.~1561 of 11.10.2022, project code PE0000021, concession decree n.~1561 of 11.10.2022 adopted by Ministero dell'Università e della Ricerca, CUP I53C22001450006, project title ``Network 4 Energy Sustainable Transition –- NEST'', Task 8.4.4).

\bibliographystyle{apalike-ejor}
\bibliography{biblio}

@ARTICLE{FiFP21,
        AUTHOR = "D. {Fioriti} and A. {Frangioni} and D. {Poli}" ,
        TITLE = "Optimal Sizing of Energy Communities with Fair Revenue
                 Sharing and Exit Clauses: Value, Role and Business Model
                 of Aggregators and Users" ,
        JOURNAL = "Applied Energy" ,
        VOLUME = "299" ,
        PAGES = "117328" ,
        YEAR = "2021"
        }

@article{Ba16,
  title={On the core and bargaining set of a veto game},
  author={Bahel, Eric},
  journal={International Journal of Game Theory},
  volume={45},
  pages={543--566},
  year={2016},
  publisher={Springer}
}

@article{Ba19,
  title={On the properties of the nucleolus of a veto game},
  author={Bahel, Eric},
  journal={Economic Theory Bulletin},
  volume={7},
  pages={221--234},
  year={2019},
  publisher={Springer}
}

@article{MuNaPoTi88,
  title={On big boss games},
  author={Muto, Shigeo and Nakayama, Mikio and Potters, Jos and Tijs, Stef},
  journal={The economic studies quarterly},
  volume={39},
  number={4},
  pages={303--321},
  year={1988},
  publisher={JAPANESE ECONOMIC ASSOCIATION}
}

@article{PoMuti90,
  title={Bargaining set and kernel of big boss games},
  author={Potters, Jos and Muto, Shigeo and Tijs, Stephanus},
  journal={Methods of Operations Research},
  volume={60},
  pages={329--335},
  year={1990}
}

@article{ArFe97,
  title={The nucleolus and kernel of veto-rich transferable utility games},
  author={Arin, Javier and Feltkamp, Vincent},
  journal={International Journal of Game Theory},
  volume={26},
  number={1},
  pages={61--73},
  year={1997},
  publisher={Springer}
}

@article{ShSh66,
  title={Quasi-cores in a monetary economy with nonconvex preferences},
  author={Shapley, Lloyd S and Shubik, Martin},
  journal={Econometrica: Journal of the Econometric Society},
  pages={805--827},
  year={1966}
}

@article{MaPeSh79,
  title={Geometric properties of the kernel, nucleolus, and related solution concepts},
  author={Maschler, Michael and Peleg, Bezalel and Shapley, Lloyd S},
  journal={Mathematics of operations research},
  volume={4},
  number={4},
  pages={303--338},
  year={1979},
  publisher={INFORMS}
}

@incollection{Gi59,
title={Solutions to general non-zero-sum games},
author={Gillies, D. B.},
year={1959},
booktitle = {Contributions to the Theory of Games IV. Annals of Mathematics Studies},
editors = {Tucker, A. W. and Luce, R. D.},
pages={47--85},
publisher = {Princeton University Press}
}

@book{MaZaSo20,
  title={Game theory},
  author={Maschler, Michael and Zamir, Shmuel and Solan, Eilon},
  year={2020},
  publisher={Cambridge University Press}
}

@article{Br07,
  title={Null or nullifying players: the difference between the {S}hapley value and equal division solutions},
  author={van den Brink, Ren{\'e}},
  journal={Journal of Economic Theory},
  volume={136},
  number={1},
  pages={767--775},
  year={2007},
  publisher={Elsevier}
}

@article{DrFu91,
  title={Coincidence of and collinearity between game theoretic solutions},
  author={Driessen, TSH and Funaki, Yukihiko},
  journal={Operations-Research-Spektrum},
  volume={13},
  number={1},
  pages={15--30},
  year={1991},
  publisher={Springer}
}

@article{PoPoTiMu89,
  title={Clan games},
  author={Potters, Jos and Poos, Ren{\'e} and Tijs, Stef and Muto, Shigeo},
  journal={Games and Economic Behavior},
  volume={1},
  number={3},
  pages={275--293},
  year={1989},
  publisher={Elsevier}
}

@article{Sc69,
title={The Nucleolus of a Characteristic Function Game},
author={Schmeidler, David},
journal={SIAM Journal on Applied Mathematics},
volume={17},
pages={1163-1170},
year={1969}
}

@article{BeFlNg21,
  title={Finding and verifying the nucleolus of cooperative games},
  author={Benedek, M{\'a}rton and Fliege, J{\"o}rg and Nguyen, Tri-Dung},
  journal={Mathematical Programming},
  volume={190},
  number={1},
  pages={135--170},
  year={2021},
  publisher={Springer}
}

@article{DuSh84,
  title={Totally balanced games arising from controlled programming problems},
  author={Dubey, Pradeep and Shapley, Lloyd S},
  journal={Mathematical Programming},
  volume={29},
  pages={245--267},
  year={1984},
  publisher={Springer}
}

@article{KaZe82,
  title={Totally balanced games and games of flow},
  author={Kalai, Ehud and Zemel, Eitan},
  journal={Mathematics of Operations Research},
  volume={7},
  number={3},
  pages={476--478},
  year={1982},
  publisher={INFORMS}
}

@article{KaZe82b,
  title={Generalized network problems yielding totally balanced games},
  author={Kalai, Ehud and Zemel, Eitan},
  journal={Operations Research},
  volume={30},
  number={5},
  pages={998--1008},
  year={1982},
  publisher={INFORMS}
}

@article{HaMoMc19,
  title={Incentivizing prosumer coalitions with energy management using cooperative game theory},
  author={Han, Liyang and Morstyn, Thomas and McCulloch, Malcolm},
  journal={IEEE Transactions on Power Systems},
  volume={34},
  number={1},
  pages={303--313},
  year={2018},
  publisher={IEEE}
}

@ARTICLE{FiBiFrPaPo25,
  author={Fioriti, Davide and Bigi, Giancarlo and Frangioni, Antonio and Passacantando, Mauro and Poli, Davide},
  journal={IEEE Transactions on Energy Markets, Policy and Regulation}, 
  title={Fair Least Core: Efficient, Stable and Unique Game-Theoretic Reward Allocation in Energy Communities by Row-Generation}, 
  year={2025},
  volume={3},
  number={2},
  pages={170-181}
}

@article{chics2017coalitional,
  title={Coalitional game-based cost optimization of energy portfolio in smart grid communities},
  author={Chi{\c{s}}, Adriana and Koivunen, Visa},
  journal={IEEE Transactions on Smart Grid},
  volume={10},
  number={2},
  pages={1960--1970},
  year={2017},
  publisher={IEEE}
}

@article{CREMERS2023120328,
title = {Efficient methods for approximating the {S}hapley value for asset sharing in energy communities},
journal = {Applied Energy},
volume = {331},
pages = {120328},
year = {2023},
issn = {0306-2619},
author = {Sho Cremers and Valentin Robu and Peter Zhang and Merlinda Andoni and Sonam Norbu and David Flynn}
}

@article{TAN2021100453,
title = {Fair-efficient energy trading for microgrid cluster in an active distribution network},
journal = {Sustainable Energy, Grids and Networks},
volume = {26},
pages = {100453},
year = {2021},
issn = {2352-4677},
author = {Mao Tan and Yuqing Zhou and Ling Wang and Yongxin Su and Bin Duan and Rui Wang}
}

@ARTICLE{9311637,
  author={Suh, Jeongmeen and Yoon, Sung-Guk},
  journal={IEEE Access}, 
  title={Profit-Sharing Rule for Networked Microgrids Based on {M}yerson Value in Cooperative Game}, 
  year={2021},
  volume={9},
  number={},
  pages={5585-5597}
}

@article{DU2018383,
title = {A cooperative game approach for coordinating multi-microgrid operation within distribution systems},
journal = {Applied Energy},
volume = {222},
pages = {383-395},
year = {2018},
issn = {0306-2619},
author = {Yan Du and Zhiwei Wang and Guangyi Liu and Xi Chen and Haoyu Yuan and Yanli Wei and Fangxing Li}
}

@techreport{Caramizaru2019,
  title = {Energy Communities: An Overview of Energy and Social Innovation},
  author = {Caramizaru, Aura and Uihlein, Andreas},
  year = {2020},
doi = {doi/10.2760/180576},
  institution = { European Commission: Joint Research
Centre, Publications Office}
}

@article{gjorgievski2021social,
  title={Social arrangements, technical designs and impacts of energy communities: A review},
  author={Gjorgievski, Vladimir Z and Cundeva, Snezana and Georghiou, George E},
  journal={Renewable Energy},
  volume={169},
  pages={1138--1156},
  year={2021},
  publisher={Elsevier}
}

@misc{arera,
  author = {{Autorità di Regolazione per Energia Reti e Ambiente}},
  year = {accessed on 6 October 2025},
  url = {https://www.arera.it},
}

@misc{gse,
  author = {{Gestore dei Servizi Energetici}},
  year = {accessed on 6 October 2025},
  url = {https://www.gse.it},
}

@article{HaKlSoTiVe03,
  title={On the nucleolus of neighbor games},
  author={Hamers, Herbert and Klijn, Flip and Solymosi, Tam{\'a}s and Tijs, Stef and Vermeulen, Dries},
  journal={European Journal of Operational Research},
  volume={146},
  number={1},
  pages={1--18},
  year={2003},
  publisher={Elsevier}
}

@article{SoRaTi05,
  title={Computing the nucleolus of cyclic permutation games},
  author={Solymosi, Tam{\'a}s and Raghavan, Tirukkannamangai ES and Tijs, Stef},
  journal={European Journal of Operational Research},
  volume={162},
  number={1},
  pages={270--280},
  year={2005},
  publisher={Elsevier}
}

@article{NgTh16,
  title={Finding the nucleoli of large cooperative games},
  author={Nguyen, Tri-Dung and Thomas, Lyn},
  journal={European Journal of Operational Research},
  volume={248},
  number={3},
  pages={1078--1092},
  year={2016},
  publisher={Elsevier}
}

@article{AbEhLa25,
  title={Risk-sharing in energy communities},
  author={Abada, Ibrahim and Ehrenmann, Andreas and Lambin, Xavier},
  journal={European Journal of Operational Research},
  volume={322},
  number={3},
  pages={870--888},
  year={2025},
  publisher={Elsevier}
}

@misc{EC301,
  author = {Davide Fioriti},
  title = {EnergyCommunity.jl (version v0.3.1)},
  url = {https://github.com/SPSUnipi/EnergyCommunity.jl/releases/tag/v0.3.1},
  year = {2025},
}

@misc{TG200,
   author = {Davide Fioriti},
   journal = {GitHub},
   title = {{TheoryOfGames.jl (version v0.2.0)}},
   url = {https://github.com/SPSUnipi/TheoryOfGames.jl},
   year = {2025},
}

\end{document}